\documentclass[11pt]{article}
\usepackage{algpseudocode}
\usepackage{float} % fix table inside context 
\usepackage{fullpage}
\usepackage{amsfonts}
\usepackage{amsmath}
\usepackage{amssymb}
\usepackage{amsbsy}
\usepackage{amsthm}
\usepackage{mathtools}
\usepackage{epsfig}
\usepackage{graphicx}
\usepackage{colordvi}
\usepackage{graphics}
\usepackage{color}
\usepackage{bm}
\usepackage{verbatim}
\numberwithin{equation}{section}
\usepackage{caption}
\usepackage{subcaption}
\newcommand{\f}{\frac}
\newcommand{\R}{\mathbb R}
\newcommand{\C}{\mathbb C}
\newcommand{\A}{\mathcal A}

\newtheorem{theorem}{Theorem}[section]
\newtheorem{thm}{Theorem}[section]
\newtheorem{lemma}[theorem]{Lemma}%[section]
\newtheorem{definition}[theorem]{Definition}
\newtheorem{algorithm}[thm]{Algorithm}
\newtheorem{prop}[theorem]{Proposition}
\newtheorem{remark}[theorem]{Remark}

\DeclareMathOperator{\re}{Re}
\title{Accelerating the Computation of Tensor $Z$-eigenvalues}
\author{
\and
Sara Pollock
\thanks{Department of Mathematics, University of Florida, Gainesville, FL
  32611-8105 (s.pollock@ufl.edu)}
\and
Rhea Shroff
\thanks{Department of Mathematics, University of Florida, Gainesville, FL
  32611-8105 (rhea.shroff@ufl.edu)}
}
\date{\today}
\begin{document}
% -- ------------------------------------------------------------------------
\maketitle
% -- ------------------------------------------------------------------------
\begin{abstract}
Efficient solvers for tensor eigenvalue problems are important tools for the
analysis of higher-order data sets.
Here we introduce, analyze and demonstrate an extrapolation method to accelerate the
widely used shifted symmetric higher order power method for tensor
$Z$-eigenvalue problems.
We analyze the asymptotic convergence of the method, determining the range of
extrapolation parameters that induce acceleration, as well as the parameter
that gives the optimal convergence rate.  We then introduce an automated
method to dynamically approximate the optimal parameter, and demonstrate 
it's efficiency when the base iteration is run with either static or 
adaptively set shifts. Our numerical results on both even and odd order 
tensors demonstrate the theory and show we achieve our theoretically 
predicted acceleration.
\end{abstract}
% -- ------------------------------------------------------------------------
% -- ------------------------------------------------------------------------
\section{Introduction}\label{sec: intro}
% -- ------------------------------------------------------------------------

Tensor analysis has been gaining attention across mathematics and physical and data 
sciences due to the need to analyze and draw inferences from growing numbers of
higher order data sets.
Tensors are algebraic objects that define a multi-linear relationship between sets of 
algebraic objects related to a vector space, and they arise naturally in the analysis 
of data-intensive problems. 
Applications of higher-order or tensor eigenvalue problems can be seen in diverse
applications including
diffusion tensor imaging \cite{GoDe09,QWW08,SFGDFL13,Schultz_Seidel_2008,SMHKN03};
data analysis and mixture models arising in applied statistics and machine learning 
\cite{AGHKT14,SeKa12};
quantum physics and quantum geometry \cite{HQZ16,QCC18,Xiong_Liu_2020};
spectral hypergraph theory \cite{benson2019,BeGlLe15,CQZ13,LQY13,Qi_Luo_2017, XiCh13}; and
high-order Markov chains and multilinear PageRank \cite{CRZT20p,GLY15}, to name a few.

Like matrices, tensors as mathematical objects are more than their coordinate 
representation. 
We can and will use the multidimensional matrix representation to define and compute tensor eigenvalues \cite{Qi_2007}. 
In particular, we will use the following definition for tensors from \cite{ Kolda_Mayo_2011, Qi_2005}.

\begin{definition}
A real $m^{th}$ order $n$ dimensional tensor $\mathcal{A}$ can be represented by $n^m$ real entries $A_{i_1,...,i_m} \in \mathbb{R}$ where $i_j = 1,...,n$ for $j = 1,...,m$. 
\label{def: tensor}
\end{definition}
Here, $\mathcal{A} \in \mathbb{R}^{n \times n ..... \times n} \text{ or } \R^{[m,n]}$ and 
$m$ is the number of modes where each is $n$ dimensional. When the dimensions for each 
mode are the same the tensor is called \textit{square}; otherwise, it is called 
\textit{rectangular}. A square tensor $\A \in \R ^{[m,n]}$ is further called symmetric 
if 
\begin{align}\label{def: symmetric tensor}
a _{i_{p(1)}...i_{p(m)}} = a _{i_1 ... i_m} \quad \textit{ for all } 
\quad p \in \Pi_m \quad \textit{ and } \quad i_1,...i_m \in \{ 1, ... , n\},
\end{align} 
where $ \Pi_m = $  set of all permutations of $(1, ..., m )$.
In this paper, we limit our discussion to symmetric tensors.

The main contribution of this paper is the introduction of an algorithm that 
accelerates the shifted symmetric higher-order power method (S-SHOPM) of
\cite{Kolda_Mayo_2011} for computing solutions to a class of tensor eigenvalue problems
known as $Z$-eigenvalues \cite{Qi_2005}. Tensor $Z$-eigenpairs are
scalar-vector pairs $(\lambda,x) \in \R \times \R^n$ that satisfy
\begin{align}\label{eqn:zeigprob}
\A x^{m-1} = \lambda x, \quad \text{ and } \quad x^Tx = 1,
\end{align}
where the tensor-vector multiplication is defined for symmetric 
$\A \in \R^{[m ,n]}$ and $x \in \R^n$ by
\begin{align}\label{def: tensor vec mult}
(\A x^{m-r})_{i_1, \ldots, i_r} \equiv \sum _{i_{r+1},\ldots,i_m} 
a_{i_1\dots i_m}x_{i_{r+1}} \ldots x_{i_m},
\end{align}  
for all $i_1, \ldots, i_r \in \{ 1,\ldots,n\}$ and $r \in \{ 0,\dots, m-1\}$.  
The definition of $Z$-eigenvalues agrees with the definition of $l^p$ eigenvalues 
from \cite{Lim_2005} for $p=2$ \cite{Cipolla_Redivo-Zaglia_Tudisco_2020_lp_extrap}.

\begin{remark}\label{rem:tenseig} 
From \eqref{eqn:zeigprob}, for $m$ even, $(\lambda,-x)$ is a $Z$-eigenpair whenever 
$(\lambda,x)$ is a $Z$-eigenpair, and for $m$ odd, $(-\lambda,-x)$ is a $Z$-eigenpair 
whenever $(\lambda,x)$ is a $Z$-eigenpair \cite{Kolda_Mayo_2011}; 
we do not consider these pairs to be distinct.
$Z$-eigenpairs as defined in \eqref{eqn:zeigprob} are a subset of the $E$-eigenpairs
which are pairs $(\lambda,x) \in \C \times \C^n$ with $Ax^{m-1} = \lambda x$ and 
$x^H x = 1$, where $x^H$ is the conjugate transpose of $x$. It was shown in 
\cite{CaSt13} that a generic symmetric tensor $\A \in \R^{[m,n]}$ has
$((m-1)^n -1)/(m-2)$ distinct $E$-eigenvalue classes; hence this number provides an upper 
bound on the number of distinct $Z$-eigenpairs.  
\end{remark}

Analogous to (shifted) power iterations for matrices, the S-SHOPM generates a 
sequence of approximate $Z$-eigenvectors by repeated tensor-vector multiplications,
shifts by a given parameter $\alpha$, and subsequent normalizations.
The choice of $\alpha$ is described in detail
in \cite{Kolda_Mayo_2011}, and briefly summarized here in section \ref{sec:bgtheory}.

\begin{algorithm}[S-SHOPM]
    \label{alg:sshopm}
    {Given a symmetric tensor $\A \in \R^{[m, n]}$ and $\alpha \in \R$, and
            $x_0 \in \R^n$ with $\|x_0\| = 1$}
    \begin{algorithmic}[1]
    \State{Let $\chi = 1$, if $\alpha \geq 0$; and $\chi = -1$, otherwise}
    \State{$\lambda_0 \gets \A x_0^m$.}
    \For{k $= 0, 1,\ldots$}
    \State{${v}_{k+1} \gets \chi (\A x_k^{m-1} + \alpha x_k)$}
    \State{$x_{k+1} \gets {v}_{k+1} / || {v}_{k+1} ||$}
    \State{$\lambda_{k+1} \gets \A x_{k+1}^m$}
    \EndFor
    \end{algorithmic}
\end{algorithm}

The original S-SHOPM was improved with the introduction of an adaptively-shifted
method in \cite{Kolda_Mayo_2014}, which substantially reduces the number of iterations 
for convergence.
Herein, we introduce an acceleration to the S-SHOPM by a one-step extrapolation.
We demonstrate analytically and numerically that this method, which has low
per-iteration complexity, reduces the asymptotic convergence rate of the iteration, hence
accelerates convergence.
Additionally, we introduce an algorithm for automated parameter selection 
which accelerates both the the S-SHOPM with static shifts as in \cite{Kolda_Mayo_2011} 
and the adaptively shifted method of \cite{Kolda_Mayo_2014}. 

Our technique to improve convergence is a depth-1 extrapolation: at each step 
the new eigenvector approximation is defined by a linear combination of the latest and
previous fixed-point updates.
Prior to the normalization on line 5 of algorithm \ref{alg:sshopm} we compute the
update
\[
u_{k+1} = (1-\gamma) v_{k+1} + \gamma v_k,
\]
and proceed to normalize the extrapolated iterate $u_{k+1}$ to produce the 
normalized eigenvector approximation $x_{k+1}$.  Our analysis 
includes the choice of optimal extrapolation parameter $\gamma$ and its relation 
to the shifting parameter $\alpha$. 
In section \ref{sec:dynamic}, we introduce and demonstrate an automated strategy
to set the extrapolation parameter that works with both a constant shift $\alpha$
and an adaptively updated shift $\alpha_k$ as introduced in \cite{Kolda_Mayo_2014}.

The underlying theory for computing $Z$-eigenpairs comes from the work of Kofidis and 
Regalia \cite{Kofidis_Regalia_2002} 
wherein they formulated the higher order power method (HOPM).  
For symmetric square tensors, this method is known as symmetric higher order power method 
(SHOPM), which is comparable to the well known power method for the matrices. 
As shown in \cite{Kofidis_Regalia_2002}, this method, in general, does 
not always converge.  However, under certain convexity conditions, 
the SHOPM is guaranteed to converge for even order tensors. 

In \cite{Kolda_Mayo_2011}, Kolda and Mayo proposed the S-SHOPM, as given here
in algorithm \ref{alg:sshopm}. 
From each initial vector used to start the iterative method, the S-SHOPM
guarantees convergence to an eigenvalue and corresponding eigenvector of a 
symmetric tensor of either odd or even order. 
In contrast to the shifted power iteration for matrices for which the algorithm
necessarily converges to the largest magnitude eigenvalue of the shifted matrix; 
for tensors, each eigenpair
has a distinct region of convergence over the unit sphere.  Hence even if only the
dominant eigenpair is sought, simulations generally consist of a substantial number
of runs from different starting vectors, and fast convergence for all of the 
eigenpairs is essential for efficiency.

If shifts are chosen large enough to ensure satisfaction of the convexity conditions,
however, the S-SHOPM can be slow to converge, 
as the asymptotically linear convergence rate depends on the shift.
For larger tensors the increased number of iterations to convergence can be problematic
due to the computational complexity of $\mathcal{O}(n^m)$ for each tensor-vector product
$\A x^{m-1}$, as given by \eqref{def: tensor vec mult}.

The proposed extrapolation method is based on one used
to accelerate the standard power iteration for matrix eigenvalue problems in 
\cite{Nigam_Pollock_2021}, and similarly, to accelerate the Arnoldi method in 
\cite{PoSc21}. 
In contrast to the matrix setting where a dynamically assigned extrapolation parameter
was found effective, here we found that a constant extrapolation 
parameter gives both better performance and has theoretical justification.  
The analysis of the presently proposed method has little in common with that presented
in \cite{Nigam_Pollock_2021,PoSc21}, both of which relied on the linear independence of
each eigenmode, which cannot be assumed for tensors.

Extrapolation methods have already been introduced to
accelerate tensor computations, for instance nonlinear GMRES for tensor
Tucker decomposition in \cite{Sterck_2012}, Nesterov acceleration for canonical
tensor decomposition in \cite{MYS20}, and simplified topological $\epsilon$-algorithms 
for $l^p$ tensor eigenvalue problems in 
\cite{Cipolla_Redivo-Zaglia_Tudisco_2020_lp_extrap}.
The convergence and acceleration properties of extrapolation methods can however be 
challenging to analyze, and the results presented herein are the first to our knowledge
to establish accelerated convergence rates for tensor eigenvalue problems theoretically
as well as computationally. In our approach, we take advantage of the fixed-point 
formulation of the S-SHOPM in \cite{Kolda_Mayo_2011} and are able to generalize the 
results to our extrapolation method
by considering the spectral radius of an augmented Jacobian matrix. 

The remainder of paper is structured as follows. 
In section \ref{sec:bgtheory} 
we review relevant background theory on the convergence of S-SHOPM which forms 
the basis of our analysis.
In section \ref{sec:extrap}, 
we state the extrapolated algorithm, and analyze its convergence
properties including the determination of an optimal extrapolation parameter in 
the main theoretical result, theorem \ref{thm:acc}.
In section \ref{sec:dynamic}, we introduce and demonstrate an automated strategy
to set the extrapolation parameter that works with both a constant shift $\alpha$
and a dynamically updated shift $\alpha_k$.
In section \ref{sec:numerics}, we present numerical examples illustrating the theory.

% -- ------------------------------------------------------------------------
\section{Background theory}\label{sec:bgtheory}
% -- ------------------------------------------------------------------------
The convergence and acceleration theory for the extrapolated method builds upon the
convergence theory for the S-SHOPM, as presented in \cite{Kolda_Mayo_2011}.
We next summarize the relevant results from \cite{Kolda_Mayo_2011}.

Let $\Sigma$ be the unit sphere on $\R ^n$, given by 
$ \Sigma = \{ x \in \R ^n : x^Tx = 1 \}$, 
and denote the spectral radius of a matrix $A$ by $\rho (A)$, 
the maximal magnitude of the eigenvalues of $A$.

% -- -------------------------------------------
\subsection{Fixed point theory}\label{subsec:fp} 
% -- -------------------------------------------
Fixed point analysis is integral to understanding the convergence of the S-SHOPM as well 
as the acceleration by extrapolation. 
Below we summarize some standard concepts. 

\begin{definition}
    A point $x_* \in \R^n$ is a fixed point of $ \phi : \R^n \to \R^n $ if $\phi (x_* ) = x_*$. Further, $x_*$ is an attracting fixed point if there exists $\delta > 0$ such that the sequence $\{ x_n\} $ defined by $x_{k+1} = \phi (x_k) $ converges to $x_*$ for any $x_0$ such that $\| x_0 - x_* \| \leq \delta$. 
    \label{def: attract fxd pt}
\end{definition}

\begin{theorem}{\cite[Theorem 2.8]{Rheinboldt}}\label{thm:afp}
    Let $x_* \in \R^n$ be a fixed point of $\phi : \R^n \to \R^n$ and let $J : \R^n \to \R^{n \times n}$ be the Jacobian of $\phi$. Then $x_*$ is an attracting fixed point if 
$\rho ( J ( x_*)) < 1$; further, if $\rho( J ( x_*)) > 0$, then the convergence of the fixed point iteration to $x_*$ is linear with rate $\rho( J ( x_*))$. 
\end{theorem}

\begin{theorem}{\cite[Theorem 1.3.7]{Stuart_1998}}\label{thm:afu}
    Let $x_* \in \R^n$ be a fixed point of $\phi : \R^n \to \R^n$, and let $J : \R^n \to \R ^{n \times n}$ be the Jacobian of $\phi$. Then $x_*$ is an unstable fixed point if 
$\rho ( J ( x_*)) > 1$. 
\end{theorem} 

A smaller value of $\rho(J(x_\ast)) \in (0,1)$ in theorem \ref{thm:afp} 
indicates a faster asymptotic convergence rate.
We will quantify the improvement in convergence rate in our acceleration method by 
showing it decreases the value of the spectral radius of the Jacobian in comparison to 
the S-SHOPM run with the same shift parameter, as introduced next.

% -- -------------------------------------------
\subsection{Shifted Symmetric Higher Order Power Method (S-SHOPM)}\label{subsec:sshopm}
In \cite{Lim_2005}, the $l^2$ (or $Z$-) eigenpairs are characterized as critical points 
of the Rayleigh quotient
$x^T \A x^{m-1} =  \A x^m$ for $x \in \Sigma$.
Denoting $f(x) = \A x^m$, the SHOPM is constructed seek maxima or minima of the
Rayleigh quotient, \cite{LMV00,Kofidis_Regalia_2002, Kolda_Mayo_2011} namely
$\max_{x \in \Sigma} |f(x)|$ or $\min_{x \in \Sigma} -|f(x)|.$

However, the convergence for this iterative method depends on the convexity (or concavity) of the function $f(x)$. 
For tensors where the underlying function $f(x)$ is not convex (or concave), 
the method does not guarantee convergence.
The S-SHOPM of \cite{Kolda_Mayo_2011} enforces this convexity (or concavity) by the 
introduction of a shift term $\alpha$. The underlying function for S-SHOPM becomes 
$$\hat{f}(x) \equiv f(x) + \alpha(x^Tx)^{m/2}.$$  

The idea for the shift parameter has been proposed before \cite{erdogan2009,regalia2003}, but these differ from the above definition in the exponent for the constant term. 
Algorithm \ref{alg:sshopm} illustrates the iterative scheme for the S-SHOPM. 
An appropriate choice of $\alpha$ guarantees convergence to the eigenvalues 
using the S-SHOPM, where the particular eigenpair converged to depends on the 
starting iterate. 
As shown in \cite{Kolda_Mayo_2011}, $f(x)$ is convex for $\alpha > \beta(\A)$, 
and $f(x)$ is concave for $\alpha < - \beta (\A)$,
where 
\begin{equation}
\label{eqn:beta function}
\beta (\A) \equiv ( m -1 ) \max_{x \in \Sigma} \rho (\A x^{m-2}).    
\end{equation}
Computationally, this characterization poses a challenge because
the quantity $\beta(\A)$ is in general a priori unknown. Overestimating $\beta(\A)$ 
to choose a safe shift $\alpha$ for S-SHOPM slows the convergence, and underestimating 
$\beta(\A)$ can prevent convergence altogether.  
As we will show in section \ref{sec:extrap},
our extrapolation approach provably accelerates the S-SHOPM 
convergence for any $\alpha$ sufficient for convergence of the S-SHOPM. 

% -- -------------------------------------------
\subsection{Convergence of S-SHOPM}\label{subsec:sshopm-conv}
% -- -------------------------------------------
We next summarize 
some background on the convergence properties of S-SHOPM. 
As shown in \cite{Kolda_Mayo_2011}, 
we can use the following matrix to classify 
an eigenpair $(\lambda_\ast,x_\ast)$
as a local minimum or maximum of the Rayleigh quotient 
\[
C(\lambda_*, x_*) \equiv U^T_* ((m-1) \A x^{m-2}_* - \lambda_* I)U_* 
\in \R ^{ (n-1 ) \times (n-1)},\] 
where the columns of $U_* \in \R ^{(n-1) \times (n-1)}$ forms an orthogonal basis for 
$x_\ast^\perp$, where  $x_\ast^\perp = \{y \in \R ^n : (y^T x^\ast = 0 \}$,
the orthogonal complement to $x^\ast$. 

\begin{definition}
    Let $\A \in \R ^{[m,n]}$ be a symmetric tensor. We say an eigenpair $(\lambda , x)$ of $\A \in \R^{[m,n]}$ is positive stable if $C(\lambda , x)$ is positive definite, negative stable if $C(\lambda , x)$ is negative definite and unstable if $C(\lambda, x)$ is indefinite. 
    \label{def: eigenpair classification}
\end{definition}

The main convergence properties for the S-SHOPM are listed below. 
\begin{theorem}\cite[Theorem 4.4,Corollary 4.6]{Kolda_Mayo_2011} 
\label{thm:sshopm convergence convex}
Let $\A \in \R^{[m,n]}$ be symmetric. For $\alpha > \beta (\A)$, $(\alpha < -\beta(\A)$
the iterates $\{ \lambda_k, x_k \} $  produced by the S-SHOPM algorithm satisfy the following properties. 
\begin{itemize}
        \item[(a)] The sequence $ \{ \lambda_k \}$ is non decreasing (non increasing)
and there exists a $\lambda_*$ such that $\lambda_k \rightarrow \lambda_*$. 
        \item[(b)]  The sequence $ \{ x_k \}$ has an accumulation point. For every such accumulation point $x_*$, the pair $(\lambda_*, x_*)$ is an eigenpair of $\A$.
        \item[(c)] If $\A$ has finitely many real eigenvectors, then there exists $x_*$ such that $x_k \rightarrow x_*$.
\end{itemize} 
\end{theorem}

To understand the rate of convergence hence the efficiency of the method, we next
look at the fixed-point formulation and the spectral radius of the Jacobian.

% -- -------------------------------------------
\subsection{Fixed point formulation}\label{subsec:ss-fp}
% -- -------------------------------------------
To characterize the eigenpairs of a symmetric tensor $\A$, we consider the fixed 
point characterization as in \cite{Kolda_Mayo_2011}. 
For the convex case of the S-SHOPM ($\alpha > \beta(\A)$), 
we can express the algorithm as the fixed point 
iteration $x_{k+1} = \phi (x_k; \alpha)$, where $\phi$ is defined as 
\begin{equation}
    \label{eqn:phi function}
    \phi(x; \alpha) = \phi_1 (\phi_2(x ; \alpha)) \text{ with } \phi_1 (x) = \f{x}{(x^Tx)^{1/2}} \text{ and } \phi_2(x ; \alpha) = \A x^{m-1} + \alpha x.
\end{equation}

The Jacobian of $\phi$ as defined in \eqref{eqn:phi function} is 
$J(x; \alpha) = \phi _1 '(\phi_2 (x ; \alpha))\phi_2 '(x ; \alpha).$ 
The derivatives of $\phi_1$ and $\phi_2$ from \eqref{eqn:phi function} are
\begin{equation}
    \label{eqn: phi_1 prime and phi_2 prime}
    \phi_1 '(x) = \f{(x^T x)I - xx^T}{(x^Tx)^{3/2}}, 
    \text{ and } \phi_2 ' (x ; \alpha) = (m-1)\A x^{m-2} + \alpha I .
\end{equation}
Evaluating at eigenpair $(\lambda,x)$ yields
\begin{align}
\label{eqn:phis at the eigenpair}
 \phi_2 (x; \alpha) = (\lambda + \alpha)x,~ 
\phi_1 ' (\phi_2 (x ; \alpha )) = \f{I - xx^T}{|\lambda + \alpha|},~  
 \text{ and } \phi_2 '(x ; \alpha) = (m-1)\A x^{m-2} + \alpha I.   
\end{align}

Therefore, since $\lambda + \alpha > 0$ 
the Jacobian at the eigenpair is 
\begin{equation}\label{eqn:sshopm jacobian}
J(x; \alpha) = \f{(m-1)(\A x^{m-2} - \lambda x x^T) + \alpha(I - xx^T)}{\lambda + \alpha}.
\end{equation}
For the concave case ($\alpha < -\beta(\A)$)
we can express the algorithm as the fixed-point iteration $x_{k+1} = -\phi(x_k,\alpha)$,
with $\phi$ given by \eqref{eqn:phi function}. 
Following the steps above we find 
\begin{align*}
J(x; \alpha) 
= -\f{(m-1)(\A x^{m-2} - \lambda x x^T) + \alpha(I - xx^T)}{|\lambda + \alpha|}
= \f{(m-1)(\A x^{m-2} - \lambda x x^T) + \alpha(I - xx^T)}{\lambda + \alpha},
\end{align*} 
since $\lambda + \alpha < 0$. Hence in either case the Jacobian is given by
\eqref{eqn:sshopm jacobian}.

An analysis of the spectral radius of \eqref{eqn:sshopm jacobian} is used
in \cite{Kolda_Mayo_2011} 
to determine the following classification of the eigenpairs as fixed points, 
summarized for both convex and concave situations as follows.
\begin{theorem}\label{thm:ns}\cite[Theorem 4.8, Corollary 2.7]{Kolda_Mayo_2011}
Let $(\lambda, x)$ be an eigenpair of a symmetric tensor $\A \in \R ^{[m,n]}$. 
Assume $\alpha \in \R$ such that $\alpha > \beta (\A)$ ($\alpha < -\beta(A)$), 
where $\beta (\A)$ is defined in \eqref{eqn:beta function}. 
Let $\phi (x)$ be given by \eqref{eqn:phi function}. 
Then $(\lambda, x)$ is negative stable (positive stable) 
if and only if $x$ is a linearly attracting fixed point of $\phi$ ($-\phi$).
\end{theorem}
% -- ------------------------------------------------------------------------
\section{Extrapolation Method}\label{sec:extrap}
% -- ------------------------------------------------------------------------
Next we introduce the extrapolated S-SHOPM method (ES-SHOPM).
We will see that fixed points of ES-SHOPM agree with fixed points of S-SHOPM,
that for an appropriate choice of parameter, a fixed-point that is
linearly attracting for the S-SHOPM as in \eqref{thm:ns} will also be linearly
attracting for ES-SHOPM, and that the extrapolation parameter can be chosen
to ensure a faster linear rate of convergence.
This method starts with a single iteration of the S-SHOPM algorithm \ref{alg:sshopm},
after which extrapolated iterate is set as a linear combination of consecutive
S-SHOPM updates $v_{j}$, $j = k,k+1$. The extrapolation parameter $\gamma$ determines
the coefficient of the linear combination. 
\begin{algorithm}[ES-SHOPM]\label{alg:es-shopm}
{Given a symmetric tensor $\A \in \R^{[m, n]}$, $\alpha \in \R$, 
$\gamma \in (-1,0]$
and $x_0 \in \R^n$ with $\|x_0\| = 1$}
\begin{algorithmic}[1]
\State{Let $\chi = 1$, if $\alpha \geq 0$; and $\chi = -1$, otherwise}
\State{Compute $v_1,x_1,\lambda_1$ with a single iteration of algorithm \ref{alg:sshopm}}
\For{k = 1, 2, \ldots} 
\State{$v_{k+1} \gets \chi(\A x_{k}^{m-1} + \alpha x_{k})$}
\State{$u_{k+1} \gets (1 - \gamma) v_{k+1} + \gamma v_{k}$}
\State{$x_{k+1} \gets u_{k+1}/\|u_{k+1}\|$}
\State{$x_{k}^\gamma \gets (1 - \gamma) x_{k} + \gamma x_{k-1}$}
\State{$\lambda_{k+1} \gets (u_{k+1}, x_{k}^\gamma) / (x_{k}^\gamma, x_{k}^\gamma)$}
\EndFor
\end{algorithmic}
\end{algorithm} 
Notice that $u_{k+1}$ defined in line 5 satisfies 
$u_{k+1} = \A (x_{k}^\gamma)^{m-1}$, where $x_k^\gamma$ is given in line 7.
Line 8 then computes the Rayleigh quotient corresponding to $x_{k}^\gamma$.

In the remainder of this section we determine, given a shifting parameter $\alpha$,
for what values of extrapolation parameter $\gamma$ ES-SHOPM accelerates
convergence, which we will see is an open subset of $(-1,0)$. 
We will also determine an optimal parameter $\gamma$, given shift $\alpha$.
In the subsequent section, we show how to set $\gamma$ dynamically, with either
constant or dynamically updated shifts. 

% -- -------------------------------------------
\subsection{Fixed-point formulation for the ES-SHOPM}\label{subsec:convextrap}
% -- -------------------------------------------
In order to construct the Jacobian for ES-SHOPM, we first write down a fixed point 
formulation for this method.
The update step for ES-SHOPM is dependent on the last two iterations. 
As such, we formulate the input to the fixed point problem as a tuple of the previous 
two iterations by 
\[
\begin{pmatrix} x_{k+1} \\ x_{k}\end{pmatrix} = \bar \phi \left( \begin{pmatrix} x_{k} 
\\ x_{k-1}\end{pmatrix} \right),
\]
where $\bar \phi$ is the function representing the extrapolated update.

Since for a bounded extrapolation parameter $\gamma$, convergence for the  
$\{x_k\}$ guarantees the convergence for the $\{ x_k ^\gamma\}$, we will formulate the 
fixed point method for the $\{x_k\}$. 
In the convex (negative stable) case the update step for $x_{k+1}$ is given by 
\begin{align}
    \label{eqn: Update for x_k ,not tuple}
    x_{k+1} & = \phi_1 (u_{k+1}) 
     = \phi_1((1 - \gamma_k)v_{k+1} + \gamma_k v_k) 
     = \phi_1 ((1- \gamma_k) \phi_2(x_k, \alpha) + \gamma_k \phi_2(x_{k-1}, \alpha )),
\end{align}
where $\phi_1$ and $\phi_2$ are defined as in \eqref{eqn:phi function}.
For the concave (positive stable) case, 
$\phi_2$ is replaced by $-\phi_2$, as in 
subsection \ref{subsec:ss-fp}, noting by the
parity of $\phi_1$ in \eqref{eqn:phi function} that 
$-\phi(x;\alpha) = \phi_1(-\phi_2(x;\alpha))$.

In order to formulate \eqref{eqn: Update for x_k ,not tuple} as a fixed-point operation,
 we think of the iterates as tuples and we 
exchange $\phi_1$ and $\phi_2$ for $\bar \phi_1$ and $\bar \phi_2$, as follows:
\begin{equation}
    \label{eqn:update for x_k tuple}
    \begin{split}
        \bar{\phi}\begin{pmatrix}
            x_{k+1} \\
            x_k
        \end{pmatrix} = \bar{\phi_1} \begin{pmatrix}
            \bar{\phi_2} \begin{pmatrix}
            \begin{pmatrix}
                x_k \\
                x_{k-1}
            \end{pmatrix};\alpha
            \end{pmatrix}
        \end{pmatrix},
    \end{split}
\end{equation}
where for $x, y \in \R^n$
\begin{equation}
    \label{eqn: phi_1 modified}
    \bar{\phi_1} \begin{pmatrix}
        \begin{pmatrix}
            x\\
            y
        \end{pmatrix}
    \end{pmatrix} = \begin{pmatrix}
        \phi_1 ( x) \\
        \phi_1 (y)
    \end{pmatrix}, 
\end{equation}
and
\begin{equation}
    \label{eqn: modified phi_2}
    \bar{\phi_2}\begin{pmatrix}
        \begin{pmatrix}
            x\\
            y
        \end{pmatrix}; \alpha
    \end{pmatrix} = \begin{pmatrix}
        (1 - \gamma_k)\phi_2 (x; \alpha) + \gamma_k \phi_2 (y; \alpha) \\
        x
    \end{pmatrix}.
\end{equation} 

Notice here that $\bar{\phi_1}$, receives the previous iterate without any changes in the second component of the tuple as this component is already normalized. Hence, we can 
write \eqref{eqn: phi_1 modified} as 
\begin{equation}
    \label{eqn: phi_1 modified final}
    \bar{\phi_1} \begin{pmatrix}
        \begin{pmatrix}
            x\\
            y
        \end{pmatrix}
    \end{pmatrix} = \begin{pmatrix}
        \phi_1 ( x) \\
        y
    \end{pmatrix}.
\end{equation}

% -- -------------------------------------------
\subsubsection{Verifying the fixed point formulation} 
% -- -------------------------------------------
Let us first look at the fixed point formulation for the S-SHOPM. 
Suppose $(\lambda, x)$ is a negative stable eigenpair. Then
\begin{equation*}
    \begin{split}
        \phi ( x; \alpha )  = \phi_1 ( \phi_2 (x; \alpha)) 
         = \phi_1 (\A x^{m-1} + \alpha x) 
         = \phi_1 ((\lambda + \alpha)x)                                              
         = \frac{(\lambda + \alpha)x}{||(\lambda + \alpha)x ||} = x. 
    \end{split}
\end{equation*}
The last equality holds
by our choice of $\alpha$ using \eqref{eqn:beta function}, to ensure 
$\lambda + \alpha > 0$. So the fixed point problem is well defined. 
Similarly, if $(\lambda,x)$ is positive stable, we exchange $\phi$ for
$-\phi$ and
$-\phi ( x; \alpha ) = -\phi_1 ((\lambda + \alpha)x)= x$, as $-(\lambda+\alpha)$ is 
positive.
Now, for the accelerated S-SHOPM, at the same eigenpair $(\lambda, x)$, using
the same value of $\alpha$ 
we have for the negative stable case 
\begin{equation*}
    \begin{split}
        \bar{\phi}\begin{pmatrix}
            \begin{pmatrix}
            x\\ x
        \end{pmatrix}; \alpha \end{pmatrix}  = \bar{\phi_1}\begin{pmatrix}
        \bar{\phi_2} \begin{pmatrix}
        \begin{pmatrix}
        x \\ x
        \end{pmatrix}; \alpha
        \end{pmatrix}
        \end{pmatrix} 
         = \bar{\phi_1} \begin{pmatrix}
        \phi_2(x ; \alpha) \\ x
        \end{pmatrix} 
         = \begin{pmatrix}
        \phi_1(\phi_2(x ; \alpha)) \\ x
        \end{pmatrix} = \begin{pmatrix}
        x \\ x
        \end{pmatrix},
    \end{split}
\end{equation*}
and similarly for the positive stable, once the necessary changes have been made.
Hence, the accelerated fixed point is also well defined. 
Further, as shown in the next proposition, the fixed points agree between the two
methods.

\begin{prop}
\label{prop: fxd pts agree}
    Let $\A \in \R^{[m,n]}$ be a symmetric tensor. Assume $\alpha,\gamma \in \R$.
For $x \in \R^n$, $\begin{pmatrix} x \\  x \end{pmatrix}$ is a fixed point of $\bar{\phi}$ of the extrapolated method as given by \eqref{eqn:update for x_k tuple} if and only if $x$ is a fixed point of S-SHOPM $\phi$ as given by \eqref{eqn:phi function}. 
\end{prop}

\begin{proof}Let $\begin{pmatrix} y_1 \\  y_2 \end{pmatrix}$ be a fixed point of 
$\bar{\phi}$ from \eqref{eqn:update for x_k tuple}, where $y_1, y_2 \in \R^n$. 
That is, 
for the negative stable case we have
\[\begin{pmatrix} y_1 \\ y_2 \end{pmatrix} 
= \bar{\phi} \begin{pmatrix} y_1 \\ y_2 \end{pmatrix} 
= \bar{\phi_1} \begin{pmatrix} \bar{\phi_2} 
\begin{pmatrix} y_1 \\ y_2 \end{pmatrix} \end{pmatrix}.\] 

Using \eqref{eqn: phi_1 modified final} and \eqref{eqn: modified phi_2} we have
\[\begin{pmatrix} y_1 \\ y_2 \end{pmatrix} 
= \bar{\phi_1} \begin{pmatrix} (1 - \gamma)\phi_2 (y_1) + \gamma \phi_2(y_2) 
\\ y_1  \end{pmatrix} 
= \begin{pmatrix} \phi_1((1 - \gamma)\phi_2 (y_1) + \gamma \phi_2(y_2)) 
\\ y_1 \end{pmatrix},
\]
From the second component we have
$y_2 = y_1$, by which the first component satisfies 
$y_1 = \phi_1 (\phi_2 (y_1))$, i.e., $y_1$ is a fixed point of $\phi$ from 
\eqref{eqn:update for x_k tuple}. 
For the positive stable case we exchange $\phi(x)$ for $-\phi(x) = \phi_1(-\phi_2(x))$,
and the conclusion follows.
We have already verified the converse, i.e, if $x \in \R^n$ is a fixed point of the 
function $\phi$, then 
$\begin{pmatrix} x \\ x \end{pmatrix}$ is a fixed point of the function $\bar{\phi}$.
\end{proof}

As we will numerically demonstrate in section \ref{sec:numerics}, the accelerated method
converges to the same set of eigenvalues from the same set of initial iterates, as 
compared to S-SHOPM using the same shift. Moreover, each eigenvalue is located the same
number of times from the same set of randomly generated initial iterates, 
suggesting (although we do not prove it here) that the basins of 
attraction are the same between the two methods.

% -- -------------------------------------------
\subsubsection{The Jacobian for the extrapolation method}
% -- -------------------------------------------
Next we will write the Jacobian of the extrapolated formulation 
\eqref{eqn:update for x_k tuple} at a solution
in terms of the Jacobian for the S-SHOPM evaluated at eigenvector $x$ of $\A$,
as given by \eqref{eqn:sshopm jacobian} 
\begin{lemma}\label{lem:accjac}
At eigenvector $x$ of $\A$, the Jacobian for the accelerated method with parameter 
$\gamma \in \R$ as
defined by \eqref{eqn:update for x_k tuple} is given by 
\begin{align}\label{eqn:accjac}
{J_\gamma} \begin{pmatrix} \begin{pmatrix} x \\ x\end{pmatrix}; \alpha \end{pmatrix} = \begin{pmatrix} (1 - \gamma) J(x; \alpha) & \gamma J(x;\alpha) \\ I_n & 0_n \end{pmatrix},
\end{align}
where $J(x;\alpha)$ is the Jacobian for the S-SHOPM with parameter $\alpha$, evaluated
at eigenpair $(\lambda,x)$, as given by \eqref{eqn:sshopm jacobian}.
\end{lemma}
\begin{proof}
The proof proceeds by direct calculation, 
which we show here explicitly for the negative stable case.
Applying the chain rule for derivatives, 
the Jacobian for the accelerated method defined by \eqref{eqn:update for x_k tuple}
can be defined as 
\begin{equation}
    \label{eqn:jacobian for accl method}
     J_\gamma \begin{pmatrix}
        \begin{pmatrix}
            x_{k+1} \\
            x_k
        \end{pmatrix} ; \alpha
    \end{pmatrix}
         = \bar{\phi_1}' \begin{pmatrix}
            \bar{\phi_2} \begin{pmatrix}
            \begin{pmatrix}
                x_k \\
                x_{k-1}
            \end{pmatrix};\alpha
            \end{pmatrix}
        \end{pmatrix} \bar{\phi_2}' \begin{pmatrix}
            \begin{pmatrix}
                x_k \\
                x_{k-1}
            \end{pmatrix};\alpha
            \end{pmatrix},
\end{equation}
where $\bar\phi_1$ and $\bar \phi_2$ are given respectively by 
\eqref{eqn: phi_1 modified final} and \eqref{eqn: modified phi_2}.

For arbitrary $x,y \in \R^n$ we can use the previously calculated 
$\phi_1 '$ and $\phi_2 '$ from the \eqref{eqn: phi_1 prime and phi_2 prime} 
to obtain
\begin{equation}
    \label{eqn: modified phi_1 prime}
    \bar{\phi_1}' \begin{pmatrix}
        \begin{pmatrix}
            x \\
            y
        \end{pmatrix}
    \end{pmatrix} = \begin{pmatrix}
        \f{\partial \phi_1 (x)}{\partial x} & \f{\partial \phi_1 (x)}{\partial y} \\
        \f{\partial y}{\partial x} & \f{\partial y}{\partial y}
    \end{pmatrix} = \begin{pmatrix}
        \f{(x^Tx)I_n - xx^T}{(x^Tx)^{3/2}} & 0_n \\
        0_n & I_n 
    \end{pmatrix}, 
\end{equation} 
and 
       
\begin{align} \label{eqn: modified phi_2 prime}
\bar{\phi_2} ' \begin{pmatrix}
        \begin{pmatrix}  x \\ y \end{pmatrix}; \alpha
    \end{pmatrix}  & = \begin{pmatrix}
        \f{\partial [ (1 - \gamma_k) \phi_2(x) + \gamma_k \phi_2 (y)]}{\partial x} & \f{\partial [ (1 - \gamma_k) \phi_2(x) + \gamma_k \phi_2 (y)]}{\partial y} \\
        \f{\partial x}{\partial x} & \f{\partial x}{\partial y}
    \end{pmatrix}
\nonumber \\
    & = \begin{pmatrix}
        (1 - \gamma_k)[(m-1)\A x^{m-2} + \alpha I] & \gamma_k[(m-1)\A y^{m-2} + \alpha I] \\
        I_n & 0_n 
    \end{pmatrix}.
\end{align} 

At the eigenpair $(\lambda, x)$, the fixed point is 
$\begin{pmatrix} x \\ x \end{pmatrix}$. 
Applying \eqref{eqn:phis at the eigenpair} to the first entry yields
\begin{equation}
\label{eqn: phi_1 bar prime at phi_2 bar}
    \bar{\phi_1}' \begin{pmatrix}
    \bar{\phi_2} \begin{pmatrix}
    \begin{pmatrix}
    x \\ x
    \end{pmatrix}; \alpha
    \end{pmatrix}
    \end{pmatrix} = \begin{pmatrix}
    \phi_1 ' (\phi_2 (x; \alpha) ) & 0_n \\ 0_n & I_n  \end{pmatrix} = \begin{pmatrix}
    \frac{(I - xx^T)}{\lambda + \alpha} & 0_n \\ 0_n & I_n
    \end{pmatrix}.
\end{equation}

Putting \eqref{eqn: phi_1 bar prime at phi_2 bar} together with
\eqref{eqn:jacobian for accl method} and \eqref{eqn: modified phi_2 prime},
the Jacobian at the eigenpair becomes \begin{equation*}
    {J_\gamma} \begin{pmatrix} \begin{pmatrix} x \\ x\end{pmatrix}; \alpha \end{pmatrix} = \begin{pmatrix}  \frac{(I - xx^T)}{\lambda + \alpha} & 0_n \\ 0_n & I_n  \end{pmatrix} \begin{pmatrix}
        (1 - \gamma)[(m-1)\A x^{m-2} + \alpha I] & \gamma[(m-1)\A y^{m-2} + \alpha I] \\
        I_n & 0_n 
    \end{pmatrix}
\end{equation*} \begin{equation}
\label{eqn: Jacobian for Accl SS-HOPM}
    = \begin{pmatrix}
    (1 - \gamma) \f{(m-1)(\A x^{m-2} - \lambda x x^T) + \alpha(I - xx^T)}{\lambda + \alpha} & \gamma \f{(m-1)(\A x^{m-2} - \lambda x x^T) + \alpha(I - xx^T)}{\lambda + \alpha} \\ I_n & 0_n 
    \end{pmatrix}.
\end{equation} 

The first component of equation \eqref{eqn: Jacobian for Accl SS-HOPM} agrees with the 
Jacobian for the S-SHOPM when the parameter is $\gamma= 0$, which is to be expected. 
Moreover we can write the first to
entries of \eqref{eqn: Jacobian for Accl SS-HOPM} in terms of $J(x;\alpha)$, 
the Jacobian of the S-SHOPM from \eqref{eqn:sshopm jacobian}, yielding  
the result \eqref{eqn:accjac}.
\end{proof}

In the next section, we will use this result to determine asymptotic rates of 
convergence for algorithm \ref{alg:es-shopm} in terms of the convergence rate for 
algorithm \ref{alg:sshopm}
% -- -------------------------------------------
\subsection{Accelerated rates of convergence}
% -- -------------------------------------------
We next characterize the rate of convergence of the accelerated method
by considering the spectral radius of its Jacobian from the fixed-point formulation
\eqref{eqn:update for x_k tuple}, together with 
theorems \ref{thm:afp} and \ref{thm:afu}. 
To establish convergence we require the spectral radius of \eqref{eqn:accjac} is 
less than one, and to establish acceleration we require the spectral radius 
to be less than that of \eqref{eqn:sshopm jacobian}, the Jacobian for the S-SHOPM.
We will focus on the latter.

We will make use of the following proposition that allows us to write down eigenpairs 
of matrices of the form \eqref{eqn:accjac}.
\begin{prop}\label{prop:augpairs}
Let $(z, \mu)$, be an eigenpair of $n \times n$ matrix $J$. 
Then for arbitrary $\gamma \in \R$, the pairs 
$(v_i,a_i), ~ i = 1,2$ are eigenpairs of the augmented matrix $J_\gamma$, where
\begin{align}\label{eqn:augpairs}
J_\gamma = \begin{pmatrix} (1-\gamma)J & \gamma J \\ I & 0 \end{pmatrix}, 
~v_i = \begin{pmatrix} a_i z \\ z \end{pmatrix}, 
~\text{ and }
a_i = \frac{(1 - \gamma)\mu \pm \sqrt{((1 - \gamma)\mu)^2 + 4\gamma \mu}}{2}.
\end{align}
\end{prop}
\begin{proof}
For $a_i = 0$ multiplying through trivially yields $J v_i = a_i v_i$. 
Otherwise, multiplying through we have
\begin{align*}
J v_i = 
\begin{pmatrix} (1-\gamma)J & \gamma J \\ I & 0 \end{pmatrix} 
\begin{pmatrix} a_i z \\ z \end{pmatrix} 
= \begin{pmatrix} (a_i(1 - \gamma)\mu + \gamma \mu) z \\ a_i z \end{pmatrix} 
= a_i \begin{pmatrix} ((1 - \gamma)\mu + \frac{\gamma \mu}{a_i}) z \\ z \end{pmatrix}.
\end{align*}
 
Then $(v_i,a_i)$ is an eigenpair of $J_\gamma$ when $a_i$ satisfies 
 $(1 - \gamma)\mu + \frac{\gamma \mu}{a_i} = a_i$.
Rearranging terms yields the quadratic equation 
$ a_i^2 - (1 - \gamma)\mu a_i - \gamma \mu = 0$, with solutions given 
by $a_i$ in \eqref{eqn:augpairs}.
\end{proof}

The next key element of our main theorem on acceleration is that the Jacobian 
for S-SHOPM at a positive or negative stable fixed point is positive semi-definite.
\begin{prop}\label{prop:semidef}
Assume the hypotheses of theorem \ref{thm:ns}. 
Then $J(x;\alpha)$ as given by \eqref{eqn:sshopm jacobian}
is positive semi-definite.
\end{prop}
\begin{proof}
The proof follows that of 
\cite[theorem 4.8]{Kolda_Mayo_2011}, given for the negative stable case. 
Let $J(x;\alpha)$ be 
given by \eqref{eqn:sshopm jacobian}, the Jacobian of \eqref{eqn:phi function} 
at eigenpair $(\lambda,x)$. Matrix $J(x;\alpha)$ is symmetric, so it suffices to show
that $y^TJ(x;\alpha)y \ge 0$ for $y \in \Sigma$.  Since $x^T J(x;\alpha)x = 0$, 
consider $y^T J(x;\alpha)y > 0$ for $y\in \Sigma$ with $y^T x = 0$,
for which 
\begin{align}\label{eqn:psd001}
y^T J(x;\alpha) y  =  \frac{y^T((m-1) \A x^{m-2})y + \alpha}{\lambda + \alpha}.
\end{align}
For the positive stable case, 
$|y^T(m-1) \A x^{m-2}y| <(m-1) \rho( \A x^{m-2}) \le \beta(\A)$ and 
$\alpha > \beta(\A) \ge |\lambda|$, by the definition of $\beta$ in 
\eqref{eqn:beta function}. Putting these inequalities together yields
\[
y^TJ(x;\alpha)y 
\ge \frac{-(m-1)\rho(\A x^{m-2}) + \alpha}{\lambda + \alpha}
\ge \frac{-\beta(\A) + \alpha}{\lambda + \alpha} > 0.
\]
For the positive stable case, 
$\alpha < -\beta(\A) < 0$ still dominates $y^T(m-1)\A x^{m-2}y$
in the numerator of \eqref{eqn:psd001}, and $\lambda + \alpha < 0$ in the 
denominator so that $y^T J(x;\alpha)y > 0$.
\end{proof}

Now we can use the result of proposition \ref{prop:augpairs} to characterize 
the eigenpairs of \eqref{eqn:accjac}, the Jacobian for the ES-SHOPM, 
with respect to the eigenpairs of \eqref{eqn:sshopm jacobian}, the Jacobian for 
the S-SHOPM.
% -- ------------------------------------------------
\begin{theorem}\label{thm:acc}
Assume the hypotheses of theorem \ref{thm:ns}. 
let $J_\gamma \in R^{2n \times 2n}$ be given as in 
\eqref{eqn:augpairs}, where $J = J(x;\alpha)$, the Jacobian for the S-SHOPM at $x$
is given by \eqref{eqn:sshopm jacobian}.
Let $\rho$ be the spectral radius of $J$. Then for $\gamma \in [\gamma_{opt},0]$,
$\rho_\gamma$ the spectral radius of $J_\gamma$ is given by
\begin{align}\label{eqn:specJgam}
\rho_\gamma = \frac{(1 - \gamma)\rho + \sqrt{((1 - \gamma)\rho)^2 + 4\gamma \rho}}{2},
\end{align}
where $\gamma_{opt}$, the value of  $\gamma$ that minimizes the spectral radius of 
$J_\gamma$ is given by
\begin{align}\label{eqn:gammaopt}
\gamma_{opt} = \frac{(\rho - 2) + 2\sqrt{1 - \rho}}{\rho},
\end{align} 
yielding a spectral radius of 
$\rho_{opt}= 1-\sqrt{1-\rho}$.
\end{theorem}
% -- ------------------------------------------------
\begin{proof}
Under the given hypotheses, the spectral radius $\rho$ of $J$ satisfies $\rho \in (0,1)$.
From proposition \ref{prop:semidef}, the spectral radius $\rho$ is the maximum 
eigenvalue of $J$.  
From proposition \ref{prop:augpairs}, 
$a_i = ( (1 - \gamma)\rho \pm \sqrt{((1 - \gamma)\rho)^2 + 4\gamma \rho})/2$, 
$i = 1,2$,
are eigenvalues of $J_\gamma$.
As $\gamma$ is perturbed away from zero, the first root 
\begin{align}\label{eqn:acc001}
a_1 = \frac{(1 - \gamma)\rho + \sqrt{((1 - \gamma)\rho)^2 + 4\gamma \rho}}{2}
\end{align}
is perturbed away from $\rho$, whereas the second root $a_2$ is perturbed away from 
zero. 
To understand why we are interested in small negative values of $\gamma$, we may
consider a linear approximation of the square root term in \eqref{eqn:acc001} 
by writing
\begin{align}\label{eqn:acc002}
a_1 = \frac{\rho}{2}\left( (1 - \gamma)+ (1-\gamma)\sqrt{1 + \f{4 \gamma}{(1+\gamma)^2 \rho}}
\right)
\approx (1-\gamma)\rho + \f{\gamma}{1-\gamma}.
\end{align}
From \eqref{eqn:acc002} it is clear that as $\gamma$ is perturbed away from zero,
the spectral radius of $J_\gamma$ increases for small values of $\gamma > 0$ 
and decreases for $\gamma < 0$. 
As the same argument holds for each eigenvalue $\mu$ of $J$, we can 
see that the largest eigenvalue of $J$ is perturbed the least, so that $a_1$ as given by
\eqref{eqn:acc001} gives the spectral radius $\rho_\gamma$ of $J_\gamma$, 
for small enough negative values of $\gamma$.  

Next, we consider the range of values of $\gamma < 0$
for which the discriminant of \eqref{eqn:acc001} is non-negative to determine the
range of extrapolation parameters for which the iteration is non-oscillatory. 
The negative value of $\gamma$ for which the discriminant of \eqref{eqn:acc001}
is equal to zero is
\begin{align}\label{eqn:acc003}
\gamma_\ast = \f{(\rho -2) + 2\sqrt{1-\rho}}{\rho} 
\in (-1,0) ~\text{ for }~ \rho \in (0,1).
\end{align}
So far, this establishes $J_\gamma$ has a spectral radius $\rho_\gamma$ 
given by \eqref{eqn:specJgam} for $\gamma \in [\gamma_{\ast},0]$. 
The value of $\gamma \in [\gamma_\ast,0]$ that minimizes $\rho_\gamma$ for
$\gamma \in [\gamma_{\ast},0]$ is clearly $\gamma_\ast$, and evaluating
$\rho_\gamma$ at $\gamma = \gamma_\ast$ yields $\rho_\ast = 1-\sqrt{1-\rho}$.
For $\gamma < \gamma_{\ast}$, the discriminant of \eqref{eqn:acc001} is negative 
so that $|\rho_\gamma|^2 = -\gamma \rho$. Evaluating $-\gamma \rho$ at
$\gamma = \gamma_{\ast}$ yields $\rho_\ast^2$. Since $-\gamma \rho$ is decreasing
with respect to $\gamma$ for $\rho > 0$ we see that $\gamma_{\ast}$ minimizes 
$\rho_\gamma$, hence $\gamma_{opt} = \gamma_\ast$  in \eqref{eqn:gammaopt}
and $\rho_{opt} = \rho_\ast$, establishing the result.
\end{proof}
The results of theorem \eqref{thm:acc}, including the minimizers $\gamma_{opt}$ and
$\rho_{opt}$ as well as the form of the curve $\rho_\gamma$ as a function of $\gamma$
to the left and right of the minimizer are numerically verified in subsection 
\ref{subsec:demorates} for both an odd and even order example. 

% -- ------------------------------------------------------------------------
\section{Dynamic parameter selection}\label{sec:dynamic}
% -- ------------------------------------------------------------------------
Global convergence of S-SHOPM depends on a sufficient shift $\alpha$. However, 
choosing shifts too large in magnitude slows convergence. From \eqref{eqn:psd001}, 
$\alpha > \beta(\A)$ $(\alpha < -\beta(\A))$ ensures the spectral radius of the 
Jacobian at a negative (positive) stable eigenpair is positive, but as 
$\alpha \rightarrow \infty$, the spectral radius approaches unity.
A good choice of shift $\alpha$ is important for the extrapolated version of the 
algorithm  as well since its rate of convergence $\rho (J_\gamma)$ is a function 
of the spectral radius of the S-SHOPM Jacobian $J(x;\alpha)$.  
As it is difficult to determine an appropriate value of the shift parameter without
a priori knowledge of the spectrum, in \cite{Kolda_Mayo_2014} Kolda and Mayo proposed a 
method which adaptively updates the shift $\alpha_k$ to satisfy a local convexity
(or concavity) condition at each iteration.
This method is called the generalized eigenproblem adaptive power (GEAP) method. 

The GEAP method can be used on a more general class of eigenproblems, however we 
restrict our attention for the present to $Z$-eigenproblems. 
Algorithm \ref{alg:geap}  \cite[Algorithm 2]{Kolda_Mayo_2014} demonstrates the 
method for adaptively choosing the shift $\alpha_k$ at each iteration to ensure
the negative or positive definiteness of the Hessian of the shifted 
objective function evaluated at each iterate. 
\begin{algorithm}{$Z$-eigenpair Adaptive S-SHOPM (GEAP)} \label{alg:geap}
{Given a symmetric tensor $\A \in \R^{[m, n]}$, tolerance $\tau>0$, and
            $x_0 \in \R^n$ with $\|x_0\| = 1$}
    \begin{algorithmic}[1]
    \State{To enforce convexity, let $\chi = 1$ and for concavity, let $\chi = -1$.}
    \For{k $= 0, 1,\ldots$}
    \State{Precompute $\A x_k^{m-2}$, $\A x_k^{m-1}$, $\A x_k^{m}$}
    \State{$\lambda_k \gets \A x_k^m$}
   \State{$\alpha_k \gets \chi \max \{ 0, (\tau - \lambda_{min} (\chi m (m-1) \A x_k^{m-2}))/m \}$}
    \State{${v}_{k+1} \gets \chi(\A x_k^{m-1} + \alpha_k x_k)$}
    \State{$x_{k+1} \gets {v}_{k+1} / || {v}_{k+1} ||$}
    \EndFor
    \end{algorithmic}
\end{algorithm}

Similarly to the choice of shift, it seems unclear how to set the optimal extrapolation 
parameter $\gamma_{opt}$
given by \eqref{eqn:gammaopt} without a priori knowledge of the S-SHOPM convergence
rate, even with knowledge of the spectrum.  Fortunately, this (or a close approximation
thereof) is an observable quantity based on the either the residual convergence rate of 
S-SHOPM or more accurately by the easily computable spectral radius of its Jacobian
at each iteration.
This is how we determine parameter $\gamma_k$ at each step in the  
dynamic extrapolation for the S-SHOPM (DES-SHOPM), using the largest eigenvalue
of $J(x_{k+1}\alpha)$ in place of $\rho(x;\alpha)$ in \eqref{eqn:gammaopt}, at each 
iteration. 
Algorithm \ref{alg:de s-shopm} gives the DES-SHOPM algorithm for a dynamic extrapolation
parameter with a static choice of shift $\alpha$. 

\begin{algorithm}{(DES-SHOPM)} {Given a symmetric tensor $\A \in \R^{[m, n]}$, $\alpha \in \R$
and $x_0 \in \R^n$ with $\|x_0\| = 1$}
    \label{alg:de s-shopm}
    \begin{algorithmic}[1]
\State{Let $\chi = 1$, if $\alpha \geq 0$; and $\chi = -1$, otherwise}
\State{Compute $v_1,x_1,\lambda_1$ with a single iteration of algorithm \ref{alg:sshopm}}
\For{k = 1, 2, \ldots} 
\State{Precompute $\A x_k^{m-2}$, $\A x_k^{m-1}$, $\A x_k^{m}$}
\State{$v_{k+1} \gets \chi(\A x_{k}^{m-1} + \alpha x_{k})$}
\State{$ J_{k+1} \gets ((m-1)(\A x_{k}^{m-2} - \lambda x_{k} x_{k}^T) + \alpha(I - x_{k}x_{k}^T))/(\lambda_{k} + \alpha)$}
   \State{$\gamma_{k+1} \gets (\lambda_{max} (J_{k+1}) - 2 + 2 \re(\sqrt{1- \lambda_{max} (J_{k+1})}))/\lambda_{max} (J_{k+1})$}
\State{$u_{k+1} \gets (1 - \gamma_{k+1}) v_{k+1} + \gamma_{k+1} v_{k}$}
\State{$x_{k+1} \gets u_{k+1}/\|u_{k+1}\|$}
\State{$x_{k}^\gamma \gets (1 - \gamma_{k+1}) x_{k} + \gamma_{k+1} x_{k-1}$}
\State{$\lambda_{k+1} \gets (u_{k+1}, x_{k}^\gamma) / (x_{k}^\gamma, x_{k}^\gamma)$}
\EndFor
    \end{algorithmic}
\end{algorithm}
\begin{remark}\label{rem:DES-SHOPM rho correction}
If $\alpha$ is not strictly chosen greater than $\beta(\A)$ given by 
\eqref{eqn:beta function}, it is possible that the spectral radius of $J(x_{k},\alpha)$
exceeds unity at certain iterations (most often in the preasymptotic regime). 
To ensure the iteration is well defined under these circumstances, the assignment
of $\gamma_k$ by the local approximation of \eqref{eqn:gammaopt} is given as
$\gamma_{k+1} = (\rho(J_{k+1}) -2 + 2\re \sqrt{1 - \rho(J_{k+1}) } )/\rho(J_{k+1})$ 
in line 7 of DES-SHOPM, where $J_{k+1} = J(x_{k+1};\alpha)$ and 
$\re$ denotes the real part.
\end{remark}
The process for approximating the optimal $\gamma$ by $\gamma_k$ does not increase the 
number of tensor vector products as computing $\A x_{k}^{m-2}$ is an intermediate step
to the computation of $\A x_k^{m-1}$, likewise used to compute $\A x_k^m.$ 
However, similarly to the GEAP algorithm \ref{alg:geap} we require information on 
the spectrum, and in this case the largest magnitude eigenvalue, of a matrix of size 
$n \times n$. 

Our last algorithm combines the adaptive shift from the GEAP algorithm \ref{alg:geap}
with the dynamic choice of extrapolation from the DES-SHOPM algorithm 
\ref{alg:de s-shopm}.
The dynamic extrapolation for the GEAP algorithm (DE-GEAP) is defined by the
following modifications on the DES-SHOPM algorithm \ref{alg:de s-shopm}.
\begin{algorithm}[DE-GEAP]\label{alg:degeap}
{Given a symmetric tensor $\A \in \R^{[m, n]}$, tolerance $\tau>0$, and
            $x_0 \in \R^n$ with $\|x_0\| = 1$, run algorithm \ref{alg:de s-shopm}
with the following modifications.}
\begin{itemize}
\item Replace line 2 in DES-SHOPM algorithm \ref{alg:de s-shopm} with
\begin{algorithmic}[0]
\State{Compute $v_1, x_1, \lambda_1, \alpha_1$ with a single iteration of 
algorithm \ref{alg:geap}}
\end{algorithmic}
\item After line 11 in DES-SHOPM algorithm \ref{alg:de s-shopm} 
\begin{algorithmic}
\State{Compute $\alpha_k$ by line 5 of algorithm \ref{alg:geap}}
\end{algorithmic}
\end{itemize}
\end{algorithm}
In section \ref{sec:numerics}, we will observe that the dynamic parameter selection 
methods outperform all the static methods, and in each test the best performance is
given by the adaptive shift plus dynamic extrapolation of the DE-GEAP algorithm
\ref{alg:degeap}.

% -- ------------------------------------------------------------------------
\section{Numerical Results}\label{sec:numerics}
% -- ------------------------------------------------------------------------
In the following numerical tests we demonstrate the efficiency of the 
ES-SHOPM algorithm \ref{alg:es-shopm}, DES-SHOPM algorithm \ref{alg:de s-shopm} and
DE-GEAP algorithm \ref{alg:degeap} in comparison to their base iteration counterparts
the S-SHOPM algorithm \ref{alg:sshopm} and GEAP algorithm \ref{alg:geap}.
We consider three benchmark examples from the literature, the first two from 
\cite{Kolda_Mayo_2011} considering both odd and
even order tensors, and both convex and nonconvex cases; the third example is a larger
tensor used as an example in \cite{Cipolla_Redivo-Zaglia_Tudisco_2020_lp_extrap}.  
In subsection \ref{subsec:demorates} we
numerically verify the acceleration provided by the theoretical convergence rates of
theorem \ref{thm:acc} for the ES-SHOPM algorithm \ref{alg:es-shopm}.

For the numerical experiments, we use the \textit{Tensor Toolbox} \cite{TensorToolbox} on MATLAB version R2023a. The experiments were performed on a laptop with 12th Gen Intel Core i7-12700H (2.30 GHz) and 16.0 GB of RAM. In each of the experiments, we initialize the method by choosing a random starting point from a uniform distribution on $[-1,1]^n$. We also use a stopping criterion of $|\lambda_{k+1} - \lambda_k | < 10^{-15}$ and 1000 maximum iterations. For the characterization of the convergence rates, we plot the norms of the 
residuals, $r_k$, given for the base methods by by 
$r_k = \A x_k^{m-1} - \lambda_k x_k$, and for the extrapolated methods by
$r_k = \A (x_k^\gamma)^{m-1} - \lambda_k x_k^\gamma,$
where $(\lambda_k,x_k)$, and respectively $(\lambda_k, x_k^\gamma)$, 
are the eigenpair approximations in a given iteration for each method.
% -- -------------------------------------------
\subsection{Example 1} \cite[Example 3.6]{Kolda_Mayo_2011} Let $\A \in \R^{[3,3]}$ be a symmetric odd order tensor defined by 
% -- -------------------------------------------
\begin{align*} 
a_{111} &= -  0.1281,  & a_{112} &=  0.0516, & a_{113} &= -  0.0954,  & a_{122} &= - 0.1958, \\
a_{123} &= - 0.1790, &  a_{133} &= -  0.2676, & a_{222} &= 0.3251, & a_{223} &= 0.2513, \\
a_{233} &= 0.1773, & a_{333} &=  0.0338.  
\end{align*} 
The complete list of eigenvalues 7 distinct eigenvalues for this tensor is given in 
\cite[table 3.2]{Kolda_Mayo_2011}. As in remark \ref{rem:tenseig}, 
$(\lambda,x)$ and $(-\lambda,-x)$
are not considered distinct eigenpairs for odd order tensors. 

Figure \ref{fig:odd res conv} displays the results from running the S-SHOPM 
algorithm \ref{alg:sshopm} and ES-SHOPM algorithm \ref{alg:es-shopm} 
in the convex case. 
In each image, we have displayed the results for a few shift parameters. 
In each case for ES-SHOPM, we chose the optimal value of parameter $\gamma$ 
as given by $\gamma_{opt}$ in \eqref{eqn:gammaopt} of theorem \ref{thm:acc}.
For the results shown within each plot, each of the methods was run from the same
starting point and converged to the same eigenvalue.  
For $\lambda = 0.8730$, the starting vector was $[-0.402911, 0.903051, -0.148865]$; 
 and for $\lambda = 0.0180$, the starting vector was $[0.638048, 0.45726, -0.619523]$. 
 The results for the concave case (not shown) look similar. 

\begin{figure}[h]
    \centering
    \includegraphics[width = 0.45\textwidth]{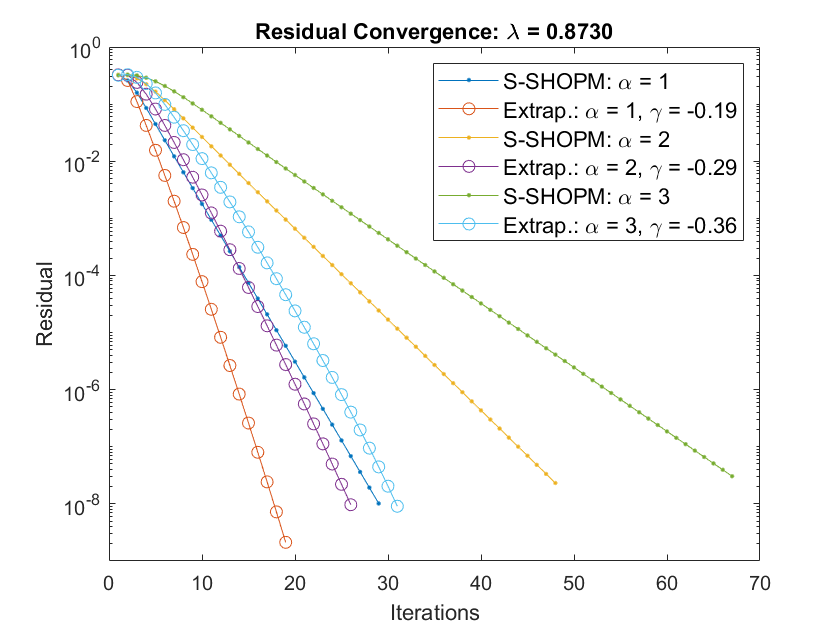}
    \includegraphics[width = 0.45\textwidth]{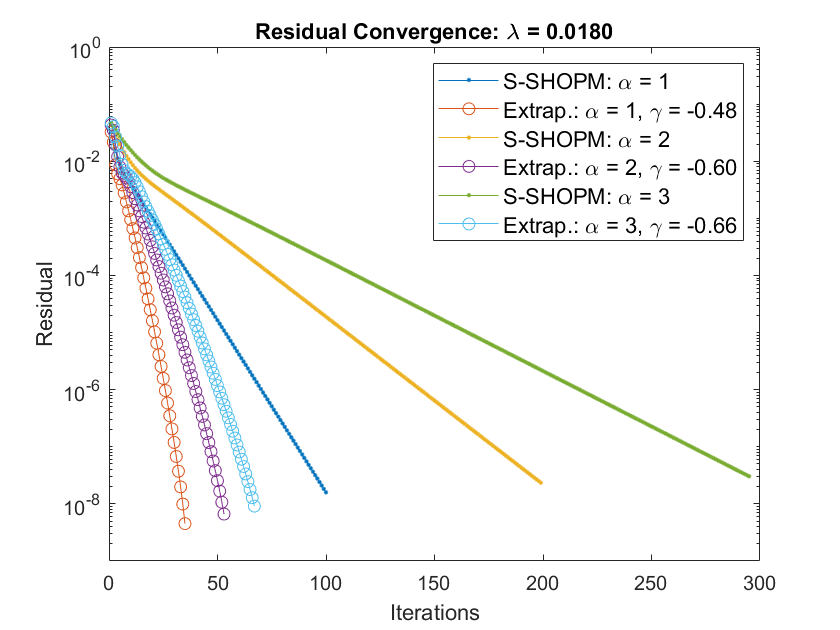}
    \caption{Example 1, residual convergence for the S-SHOPM and ES-SHOPM in the 
convex case for various $\alpha$ and corresponding $\gamma_{opt}$ as given by 
\eqref{eqn:gammaopt}: $\lambda = 0.8730$ (left) and $\lambda = 0.0180$ (right).}
\label{fig:odd res conv}
\end{figure}

Next we compare results using the dynamic extrapolation parameter 
and adaptive shift algorithms GEAP (adaptive shift), 
DES-SHOPM (dynamic extrapolation, static shift) and DE-GEAP (dynamic extrapolation,
adaptive shift) from Section \ref{sec:dynamic}. 
The results in Figure \ref{fig:de odd results} demonstrate that the dynamic methods 
converge faster 
for both convex and concave cases. 
The initial vector used for the convex case is 
$[-0.402911, 0.903051, -0.148865]$ and for the 
concave case $[-0.627312 0.38184 -0.678732]$. 
As predicted by the theory, the extrapolated algorithms with either dynamic or optimal
parameter selection converge at a better rate than their corresponding base iterations
(without acceleration); and, the fastest converging method in both cases is the DE-GEAP
using adaptive shift $\alpha_k$ and dynamically chosen $\gamma_k$. 

\begin{figure}[h]
    \centering
    \includegraphics[width = 0.45 \textwidth]{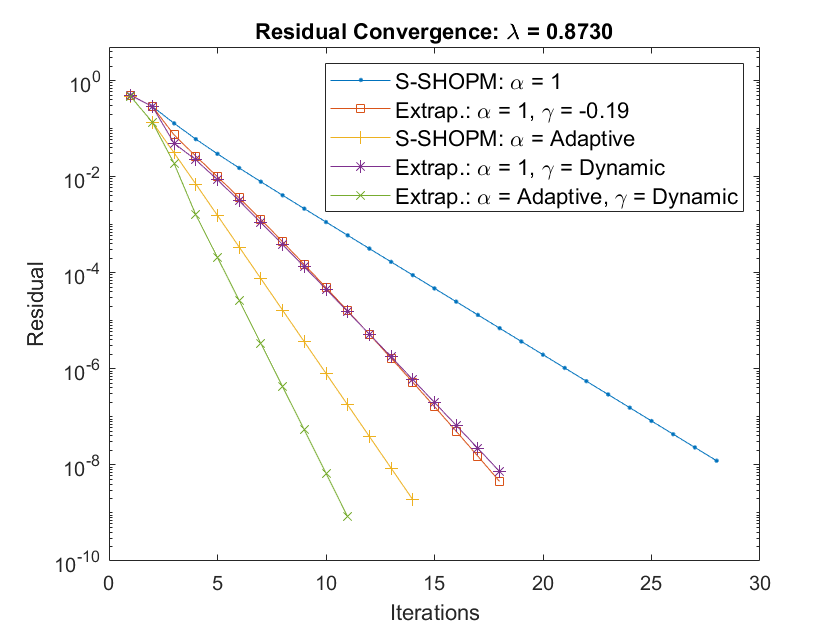}
    \includegraphics[width = 0.45 \textwidth]{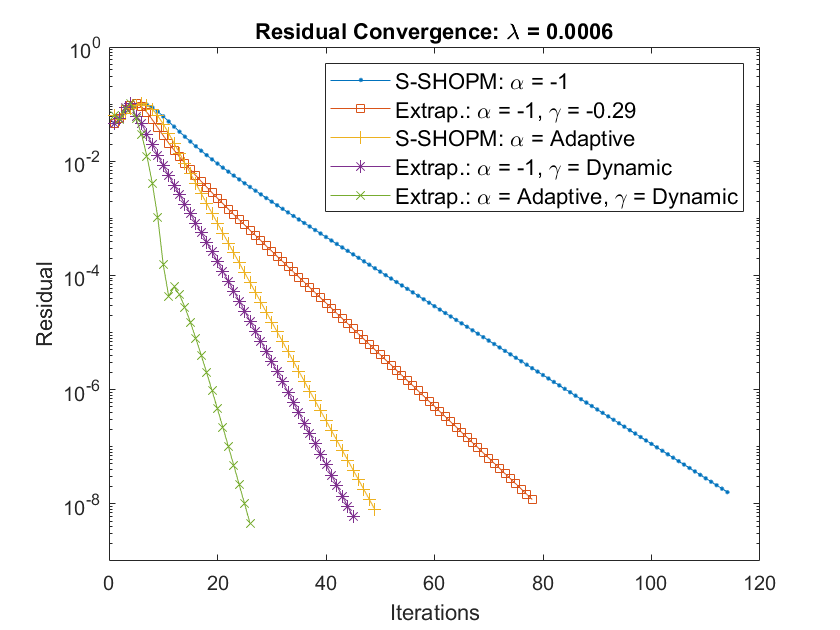}
    \caption{For Example 1, comparing residual convergence using dynamic parameter selection  : convex (left) and concave (right) case}
    \label{fig:de odd results}
\end{figure}
We also ran 1000 trials of the same experiment with randomly chosen initial vectors to observe the overall convergence patterns. For each trial, all the methods were run for the same starting vector. The results for those are summarized in table \ref{tbl: odd result compare} and table \ref{tbl: odd result compare cave}. The median iterations of the extrapolation methods are substantially less than the algorithm they're applied to. 
In particular, the dynamic extrapolation with both S-SHOPM and GEAP outperforms the original methods, respectively. Overall, DE-GEAP converges with the fewest iterations. 
We also observe that S-SHOPM and ES-SHOPM appear to have the same basins of 
attraction, i.e. they converge to the same eigenvalues for the same initial vector. 
However, for the dynamic methods, this is mostly but not strictly true as can be 
observed in the number of occurrences of each eigenvalue. 

\begin{table}[H]
    \centering
\begin{tabular}{|c||cc||cc||cc||cc||cc|}
\hline
    $\lambda$ &   \multicolumn{2}{c||}{S-SHOPM} &  \multicolumn{2}{c||}{ES-SHOPM} &  \multicolumn{2}{c||}{DES-SHOPM} &  \multicolumn{2}{c||}{GEAP} &  \multicolumn{2}{c|}{DE-GEAP}\\ 
\hline
     {} &  Its. &  \# Occ. & Its.& \# Occ. &  Its.& \# Occ.&  Its.& \# Occ.&  Its.& \# Occ.\\
\hline
    0.8730 &  29& 380 & 20& 380 & 18 & 381 & 13 & 378 & 11 & 392\\ 
\hline
    0.4306 & 47 & 300 & 24 & 300 & 25 & 299 & 24 & 300 & 16 & 303\\ 
\hline 
    0.0180 & 107&177&72&177 & 42 & 177 & 41 & 177 & 23 & 171\\
\hline
    -0.0006 & 135 & 143 & 92 & 143 & 48 & 143 & 17 & 145 & 13 & 134 \\
\hline
\end{tabular}
    \caption{Example 1, eigenvalue occurrences with $\alpha = 1$ (convex), $\gamma = -0.30$}
    \label{tbl: odd result compare}
\end{table}

\begin{table}[H]
    \centering
\begin{tabular}{|c||cc||cc||cc||cc||cc|}
\hline
    $\lambda$ &   \multicolumn{2}{c||}{S-SHOPM} &  \multicolumn{2}{c||}{ES-SHOPM} &  \multicolumn{2}{c||}{DES-SHOPM} &  \multicolumn{2}{c||}{GEAP} &  \multicolumn{2}{c|}{DE-GEAP}\\ 
\hline
     {} &  Its. &  \# Occ. & Its.& \# Occ. &  Its.& \# Occ.&  Its.& \# Occ.&  Its.& \# Occ.\\
\hline
    -0.8730 &  29 & 357 & 27 & 357 & 18 & 357 & 13 & 356 & 10 & 365 \\ 
\hline
    -0.4306 & 47 & 317 & 31 & 317 & 25 & 317 & 24 & 317 & 16 & 319 \\ 
\hline 
    -0.0180 & 107 & 180 & 36 & 180 & 41 & 180 & 41 & 180 & 22 & 175 \\
\hline
    0.0006 & 134 & 146 & 52 & 146 & 48 & 146 & 17 & 147 & 13 & 141 \\
\hline
\end{tabular}
    \caption{Example 1, eigenvalue occurrence with $\alpha = -1$ (concave), $\gamma = -0.50$}
    \label{tbl: odd result compare cave}
\end{table}

% -- -------------------------------------------
\subsection{Example 2} 
% -- -------------------------------------------
\cite[Example 1]{Kofidis_Regalia_2002} \cite[Example 3.5]{Kolda_Mayo_2011} Let $\A \in \R^{[4,3]}$ be a symmetric even order tensor defined by 
\begin{align*} 
a_{1111} &= 0.2883,  & a_{1112} &=  -0.0031, & a_{1113} &= 0.1973,  & a_{1122} &= - 0.2485, \\
a_{1123} &= - 0.2939, &  a_{1133} &=  0.3847, & a_{1222} &= 0.2972, & a_{1223} &= 0.1862, \\
a_{1233} &=  0.0919, &  a_{1333} &= - 0.3619, & a_{2222} &= 0.1241, & a_{2223} &= -0.3420, \\
a_{2233} &=  0.2127, &  a_{2333} &=  0.2727, & a_{3333} &= -0.3054. &  &
\end{align*} 

This tensor has 11 real eigenpairs and the complete list can be found in \cite[table 3.1]{Kolda_Mayo_2011}. The underlying function associated with this tensor 
is not convex, so the $\alpha$ selection is integral to convergence. 
Using the methodology as the previous example, we display an illustrative example of
residual convergence in figure \ref{fig:even res conv}, for both convex and concave cases,
Then show the number of occurrences of each eigenvalue with the median number of 
iterations for each method from 1000 different starting vectors in table
\ref{tbl: even results compare} for the convex case and table 
\ref{tbl: even results compare cave} for the concave case.

For each plot in figure \ref{fig:even res conv} we chose the same same starting point 
and the methods all converged to the same eigenpair. 
For $\lambda = 0.8893$, the starting vector was $[0.00106864, -0.0655103, -0.997851]$; 
 and for $\lambda = -1.0954$, we used $[0.10571, 0.977667, -0.18164]$.

\begin{figure}[h]
    \centering
    \includegraphics[width = 0.45 \textwidth]{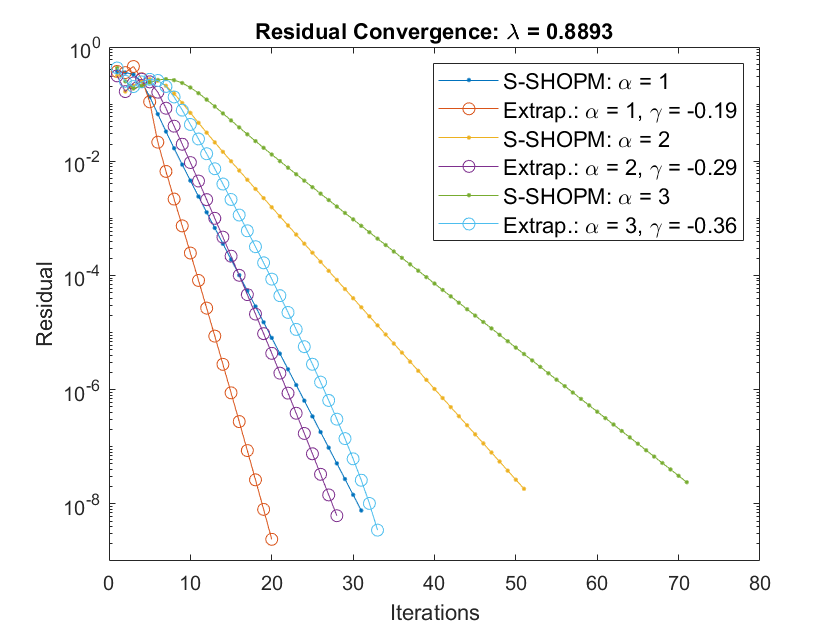}
    \includegraphics[width = 0.45 \textwidth]{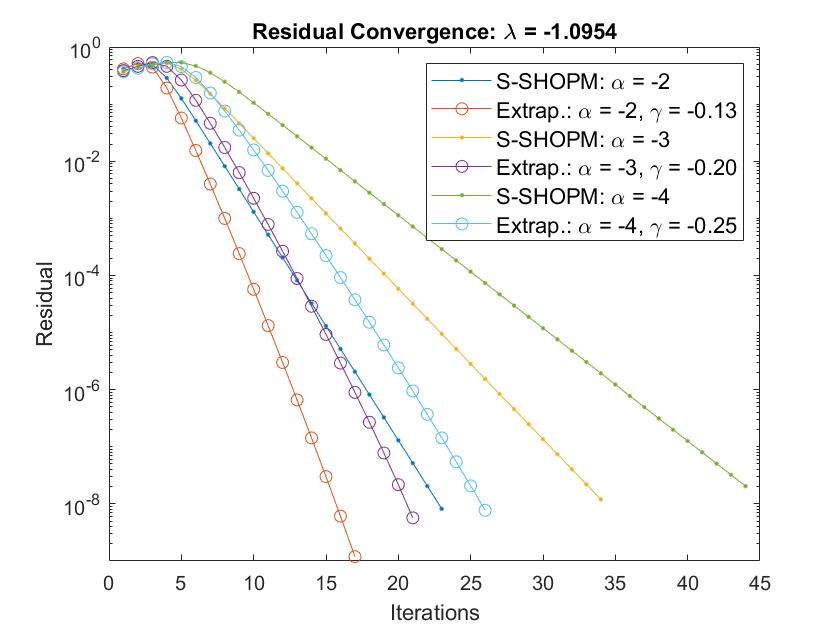}
    \caption{Example 2, residual convergence for the S-SHOPM and ES-SHOPM for various $\alpha$ and corresponding $\gamma_{opt}$ given by \eqref{eqn:gammaopt}: for the convex case, $\lambda = 0.8893$ (left) and for the concave case, $\lambda = -1.0954$ (right).}
    \label{fig:even res conv}
\end{figure}

For the dynamic parameter selection algorithm, the results are presented in figure 
\ref{fig:de even results} for the convex case. In this example we illustrate the 
performance of the algorithms on starting vectors close and far from saddle points
(which none of the algorithms will converge to). Both images show convergence to the 
same eigenpair. The right plot in figure \ref{fig:de even results} 
uses a starting vector close to a saddle point, whereas the left plot of  
figure \ref{fig:de even results} the starting vector is away from a saddle point. 
For the image on the right, the starting point is 
$[ 0.339331, -0.78868, 0.512677]$, which is close to the eigenvector, 
$[0.3598, -0.7780, 0.5150]$, associated with an unstable eigenvalue, $\lambda = 0.5105$. 
For the image on the left, the starting point is  $[0.00106864, -0.0655103, -0.997851]$, 
which is not close to any saddle point eigenvectors. 
For the plot on the right, the preasymptotic regime of the methods is prolonged, however
all methods do eventually achieve expected asymptotic rates. 
Here, the methods with dynamically chosen extrapolation parameters are seen both to 
achieve better convergence rates and to reduce the preasymptotic iterations in 
comparison to their base counterparts.  

\begin{figure}[H]
    \centering
    \includegraphics[width = 0.45 \textwidth]{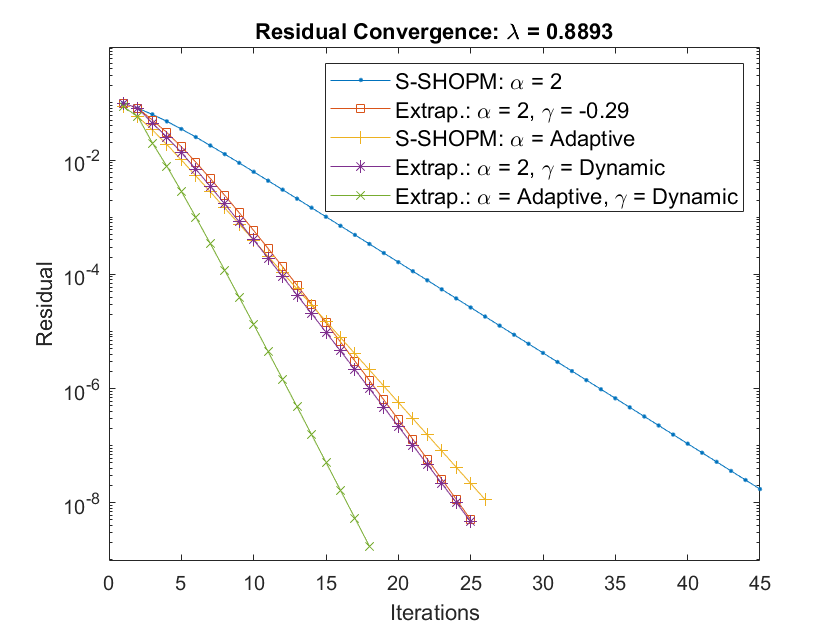}
    \includegraphics[width = 0.45 \textwidth]{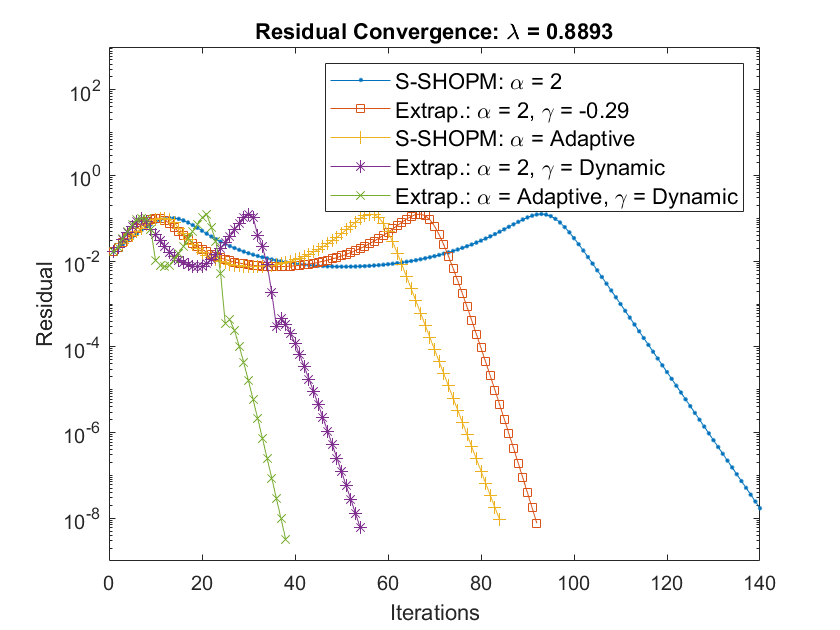}
    \caption{For Example 2, comparing residual convergence using dynamic extrapolation for the convex case with $\lambda = 0.8893$ : initial vector away from saddle point (left) and initial vector close to a saddle point (right) }
    \label{fig:de even results}
\end{figure}
We next ran 1000 trials from different starting vectors, the results of which are 
summarized in table \ref{tbl: even results compare} for the convex case,  and 
table \ref{tbl: even results compare cave} for the concave case.
 As before, we see the same convergent spectra for all the methods. 
In agreement with the theory the accelerated methods outperform the original methods, 
with the dynamic extrapolations performing the best overall. Just as in example 1, 
the basins of attraction for the various eigenpairs are not entirely the same under
adaptive shifts, although
the difference is quite small. 
\begin{table}[h]
    \centering
\begin{tabular}{|c||cc||cc||cc||cc||cc|}
\hline
    $\lambda$ &   \multicolumn{2}{c||}{S-SHOPM} &  \multicolumn{2}{c||}{ES-SHOPM} &  \multicolumn{2}{c||}{DES-SHOPM} &  \multicolumn{2}{c||}{GEAP} &  \multicolumn{2}{c|}{DE-GEAP}\\ 
\hline
     {} &  Its. &  \# Occ. & Its.& \# Occ. &  Its.& \# Occ.&  Its.& \# Occ.&  Its.& \# Occ.\\
\hline
    0.8893 & 52 & 498 & 29 & 498 & 26 & 498 & 32 & 498 & 20 & 498\\ 
\hline
    0.8169 & 45 & 303 & 26 & 303  & 24 & 303 & 34 & 302 & 20 & 304\\ 
\hline 
    0.3633 & 59 & 199 & 26 & 199 & 28 & 199 & 25 & 200 & 17 & 198\\
\hline
\end{tabular}
    \caption{For Example 2, eigenvalue occurrence with $\alpha = 2$ (convex), $\gamma = -0.35$}
    \label{tbl: even results compare}
\end{table}

\begin{table}[h]
    \centering
\begin{tabular}{|c||cc||cc||cc||cc||cc|}
\hline
    $\lambda$ &   \multicolumn{2}{c||}{S-SHOPM} &  \multicolumn{2}{c||}{ES-SHOPM} &  \multicolumn{2}{c||}{DES-SHOPM} &  \multicolumn{2}{c||}{GEAP} &  \multicolumn{2}{c|}{DE-GEAP}\\ 
\hline
     {} &  Its. &  \# Occ. & Its.& \# Occ. &  Its.& \# Occ.&  Its.& \# Occ.&  Its.& \# Occ.\\
\hline
    -0.0451 & 34 & 259 & 24 & 259 & 20 & 259 & 18 & 260 & 13 & 260 \\ 
\hline
    -0.5629 & 19 & 329 & 15 & 329 & 14 & 329 & 17 & 329 & 13 & 329 \\ 
\hline 
   -1.0954 & 20 & 412 & 15 & 412 & 15 & 412 & 17 & 411 & 13 & 411\\
\hline
\end{tabular}
    \caption{For Example 2, eigenvalue occurrence with $\alpha = -2$ (concave), $\gamma = -0.20$}
    \label{tbl: even results compare cave}
\end{table}

% -- -------------------------------------------
\subsection{Demonstrating rates of convergence}\label{subsec:demorates}
% -- -------------------------------------------
In figure \ref{fig:rate of conv}, we numerically verify that the residual rate of 
convergence for ES-SHOPM coincides with the spectral radius of $J_\gamma$ given by 
\eqref{eqn:specJgam} of theorem \ref{thm:acc}. For this numerical experiment, 
we chose a fixed starting point and shift $\alpha$ to run the ES-SHOPM algorithm
\ref{alg:es-shopm} for different values of $\gamma$. The initial vector for example 1 
(left) with $\lambda = 0.8730$ and $\alpha = 1$ is $[-0.402911, 0.903051, -0.148865] $. 
The initial vector for example 2 (right) with $\lambda = 0.3633$ and $\alpha = 2$ 
is $[0.357378, 0.670958, 0.649689] $.
We also display the expected rate of convergence by calculating the $\rho(J_\gamma)$ 
for each $\gamma$. From \cite{Kolda_Mayo_2011}, the rate of convergence of 
the S-SHOPM agrees with $\rho (J)$, where $J$ is the Jacobian from equation 
\eqref{eqn:sshopm jacobian}. Our formulation for $\rho (J_\gamma)$ maintains the
S-SHOPM rate for  $\gamma = 0$. 
Moreover, we can see from figure that for fixed values of $\alpha$ and $\lambda$ 
the algorithm attains the expected value of convergence for various choices of 
$\gamma \in [\gamma_{opt},0]$.  We do not show results for $\gamma < \gamma_{opt}$ 
because as shown in theorem \ref{thm:acc}, the eigenvalues of the Jacobian are
complex, and the convergence becomes oscillatory in that regime.

\begin{figure}[h]
    \centering
    \includegraphics[width = 0.45 \textwidth]{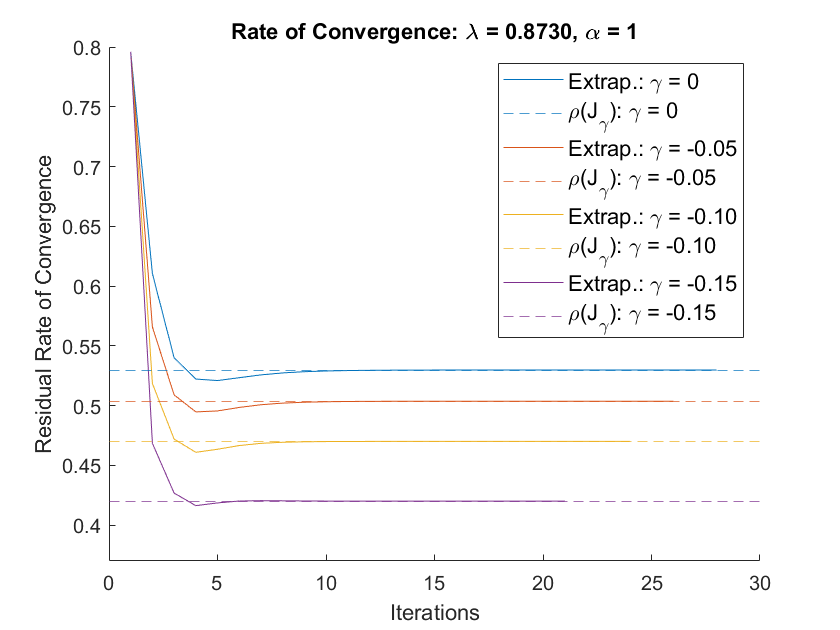}
    \includegraphics[width = 0.45 \textwidth]{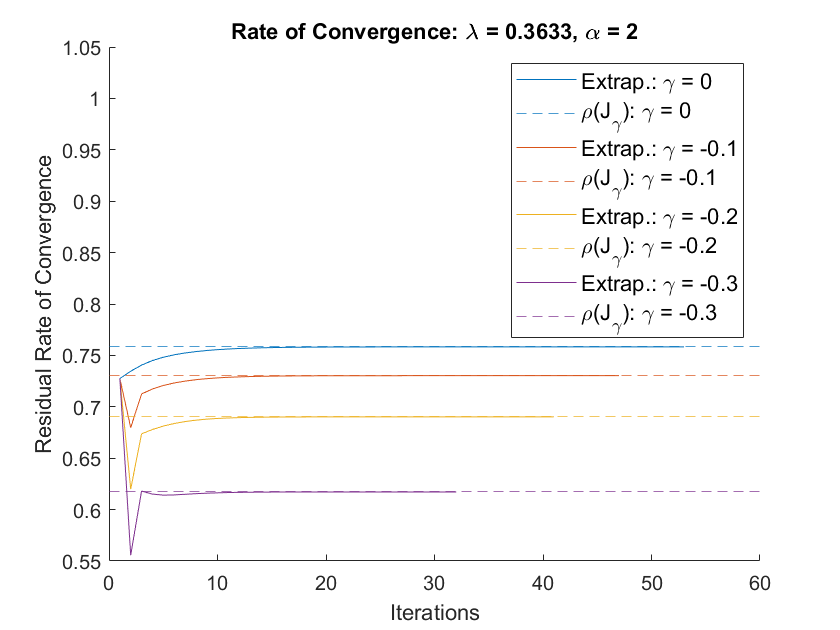}
    \caption{Rates of convergence and their estimates : for example 1 (left) and 
example 2 (right).}
    \label{fig:rate of conv}
\end{figure}

% -- -------------------------------------------
In figure \ref{fig:gamma dependence}, we verify the results of theorem \ref{thm:acc}
with for example 1 (left) and example 2 (right). Here, the computed 
spectral radius $\rho_\gamma$ of the augmented Jacobian is plotted against $\gamma$ 
for three different shift values.  The value of $\rho_\gamma$ at $\gamma = 0$ 
corresponds to the $\rho$ the spectral radius of the S-SHOPM Jacobian at the solution.
In each plot the minimizer $\gamma_{opt}$ agrees with the result of theorem \ref{thm:acc}.
To the left of the minimizer, $\gamma < \gamma_{opt}$ the curve for $\rho_\gamma$ 
agrees with $\sqrt{(-\gamma \rho)}$ and to the right, the curve agrees with 
$\rho_\gamma$ as given by \eqref{eqn:specJgam}.
\begin{figure}[h]
    \centering
    \includegraphics[width = 0.45 \textwidth]{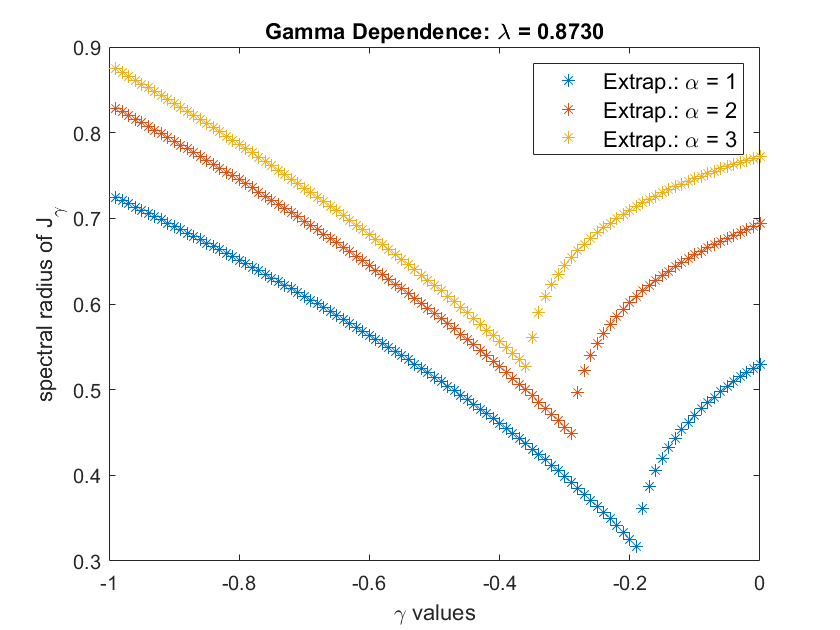}
    \includegraphics[width = 0.45 \textwidth]{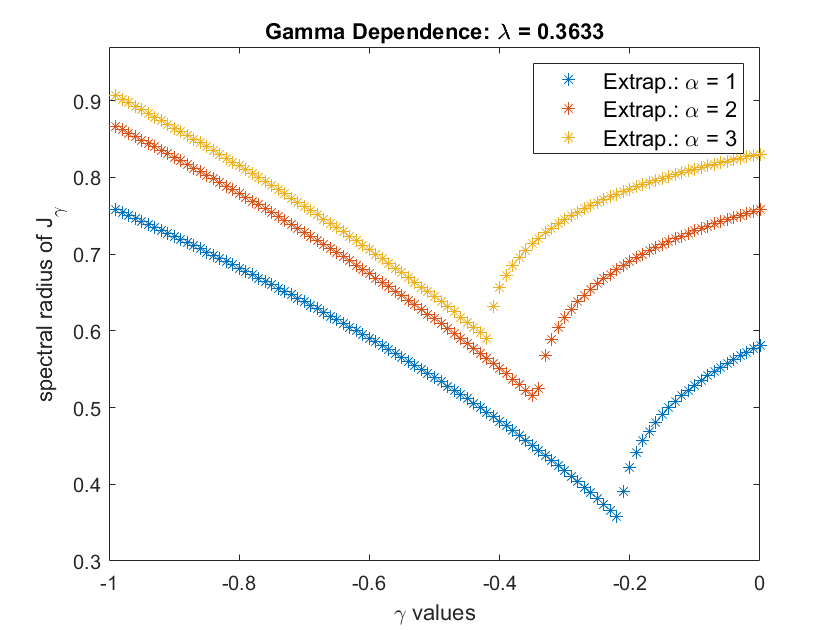}
    \caption{Spectral radius $\rho(J_\gamma)$ as a function of $\gamma$ : 
    example 1 (left) and example 2 (right). }
    \label{fig:gamma dependence}
\end{figure}

% -- -------------------------------------------
\subsection{Example 3: Dolphins Matrix}
% -- -------------------------------------------
In order to test the efficiency of our methods on a larger problem, we ran experiments 
similar to examples 1 and 2 on a larger example. 
As in \cite{Cipolla_Redivo-Zaglia_Tudisco_2020_lp_extrap} 
we constructed a tensor $\A \in \R^{[3,62]}$ 
by capturing the 3 cycles of the dolphins matrix from \cite{dolphins}, 
an undirected sparse adjacency matrix of size $62 \times 62$.

In Figure \ref{fig:dolphins results}, we compare the results of S-SHOPM and ES-SHOPM for 
different shift parameters run with their corresponding $\gamma_{opt}$ of 
\eqref{eqn:gammaopt}, 
for convex (left) and concave (right) cases. As before, the ES-SHOPM performs 
significantly better than the S-SHOPM. In Figure \ref{fig:de dolphins}, we display 
the results for the dynamic parameter selection for two instances, one convex (left) 
and one concave (right). Here there isn't a lot of difference in the performance of the 
DES-SHOPM (dynamic extrapolation, fixed shift) and the DE-GEAP (dynamic extrapolation,
adaptive shift). 
However, we do see a significant performance gain between the results run and without
extrapolation.
\begin{figure}[h]
    \centering
    \includegraphics[width = 0.45 \textwidth]{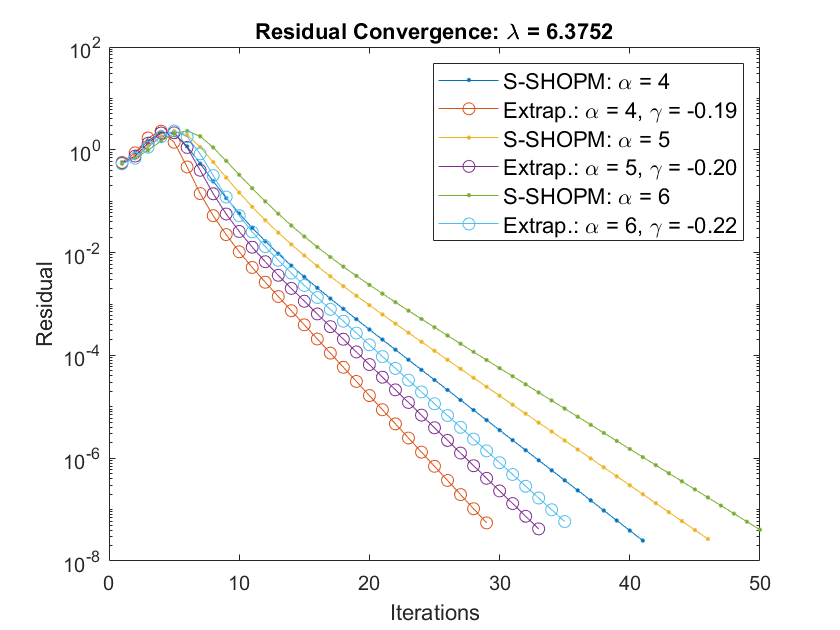}
    \includegraphics[width = 0.45 \textwidth]{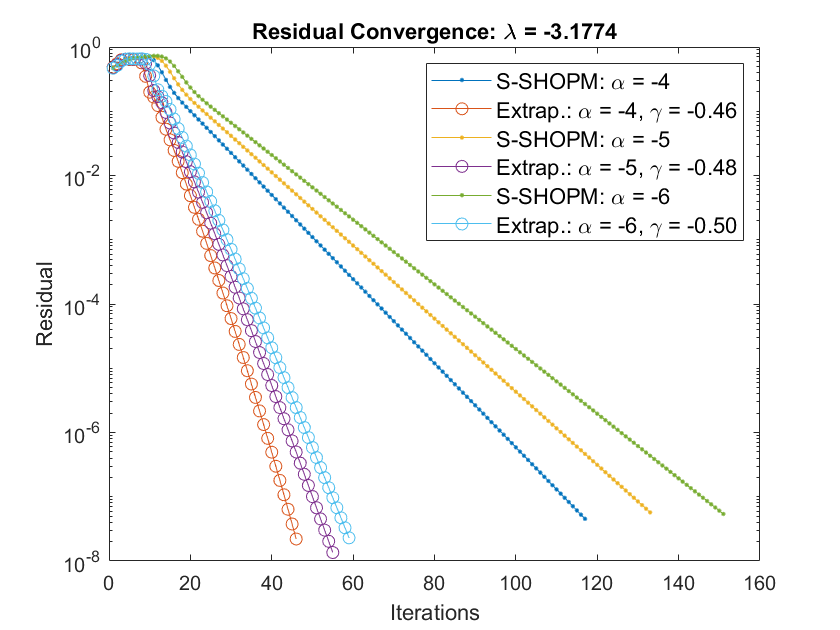}
    \caption{For dolphins matrix, residual convergence for S-SHOPM and ES-SHOPM for various $\alpha$ and corresponding $\gamma_{opt}$ : for the convex,  (left) and concave (right) cases }
    \label{fig:dolphins results}
\end{figure}
\begin{figure}[h]
    \centering
    \includegraphics[width = 0.45 \textwidth]{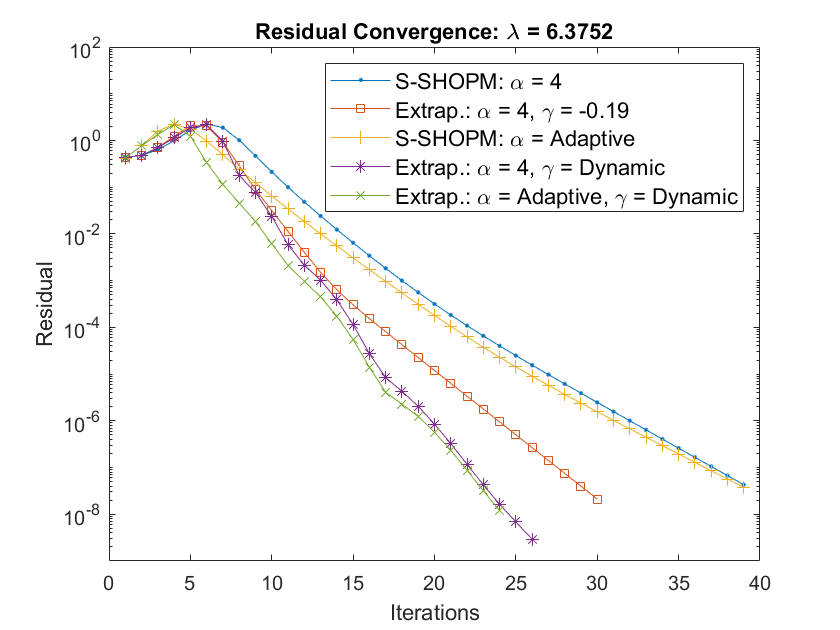}
    \includegraphics[width = 0.45 \textwidth]{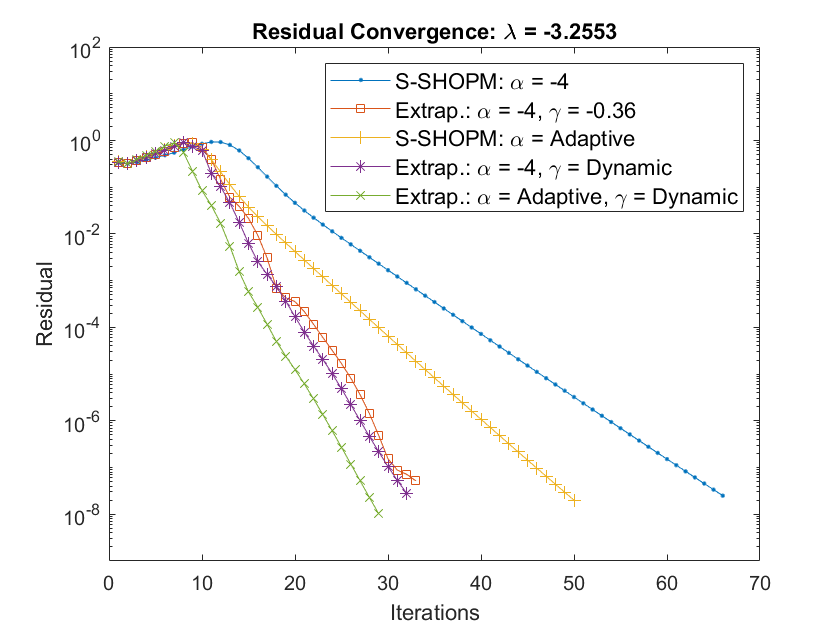}
    \caption{Example 3, comparing residual convergence using dynamic extrapolation: 
     for the convex  (left) and concave (right) cases.}
    \label{fig:de dolphins}
\end{figure}

% -- ------------------------------------------------------------------------
\section{Conclusion}\label{sec:conclusion}
% -- ------------------------------------------------------------------------
In this paper we introduced ES-SHOPM, an extrapolation algorithm to accelerate 
convergence of S-SHOPM for finding $Z$-eigenpairs of symmetric tensors.  We provided an 
analysis of the method which provides a range of extrapolation parameters for 
which ES-SHOPM provably converges at a better asymptotic rate than S-SHOPM for a 
given shift sufficient for convergence. Further, we 
derived the parameter which provides the optimal convergence rate for an extrapolation 
method of the form studied. We then introduced an automated algorithm suitable for 
either the statically shifted S-SHOPM algorithm or the adaptively shifted GEAP algorithm,
to dynamically approximate the optimal parameter.  We showed in numerical tests
that the introduced extrapolation algorithms accelerated convergence as expected 
in both statically and dynamically shifted cases. We also showed in the statically 
shifted cases that the expected rates of convergence were achieved.

In future work we will study additional extrapolation algorithms including momentum 
methods and Anderson acceleration, the latter of which has been shown in preliminary
numerical tests to accelerate convergence to unstable as well as stable
eigenvalues.  Future work may also encompass more general classes of tensor
eigenvalue problems as well as global convergence properties, including the observed 
behavior of dynamically chosen extrapolation parameters 
significantly reducing the number of preasymptotic iterations when the method is
started near an eigenvector of an unstable eigenvalue.

% -- ------------------------------------------------------------------------
% -- ------------------------------------------------------------------------
\section{Acknowledgements}
Author SP acknowledges partial support from NSF grant DMS 2045059.
% -- ------------------------------------------------------------------------
% -- ------------------------------------------------------------------------
\bibliographystyle{plain}
\bibliography{Citation}
% -- ------------------------------------------------------------------------
\end{document}